\RequirePackage{fix-cm}
\documentclass[smallextended]{svjour3}       
\smartqed  
\usepackage{graphicx}
\usepackage{amsmath,dsfont,mathrsfs,enumerate}
\usepackage{mathtools}
\usepackage{bbm}

\newcommand{\R}{\mathbb{R}}

\newcommand{\EE}{\mathbb{E}}
\newcommand{\G}{\mathcal{G}}
\newcommand{\F}{\mathcal{F}}

\newcommand{\1}{\mathbbm{1}}

\newtheorem{assumption}{Assumption}
\usepackage{amssymb}

\usepackage[author year]{natbib}  

\begin{document}

\title{The randomly distorted Choquet integrals with respect to a $\G$-randomly distorted capacity and risk measures 
}

\titlerunning{The randomly distorted Choquet integrals and risk measures}        

\author{Ohood Aldalbahi         \and
        Miryana Grigorova 
}


\institute{O. Aldalbahi\\ Department of Statistics, University of Warwick \at
              CV4 7AL, Coventry, UK\\              \email{ohood.aldalbahi@warwick.ac.uk}           
           \and
           M. Grigorova\\ Department of Statistics, University of Warwick \at
              CV4 7AL, Coventry, UK\\ 
              \email{miryana.grigorova@warwick.ac.uk} 
                \\
                https://orcid.org/0000-0002-6933-9286\\
              Corresponding Author           
}

\date{Received: date / Accepted: date}

\maketitle

\begin{abstract}
We study randomly distorted Choquet integrals with respect to a 
capacity $c$ on a measurable space $(\Omega,\mathcal F)$, where the capacity $c$ is distorted by a $\G$-measurable random distortion function (with  $\G$  a sub-$\sigma$-algebra of $\F$). 
We establish some fundamental properties, including the comonotonic additivity of these integrals under suitable assumptions on the underlying capacity space. 
We provide a representation result for comonotonic additive conditional risk measures which are monotone with respect to the first-order stochastic dominance relation (with respect to the capacity $c$) in terms of these randomly distorted Choquet integrals.
We also present the case where the random distortion functions are concave.
In this case, the $\G$-randomly distorted Choquet integrals are characterised in terms of comonotonic  additive conditional risk measures which are monotone with respect to the stop-loss stochastic dominance relation (with respect to the capacity $c$).
We provide  examples, extending some well-known risk measures in finance and insurance, such as the Value at Risk and the Average Value at Risk.

\end{abstract}

\section{Introduction}
Choquet capacities and Choquet integrals have been introduced by \cite{choquet1954theory} as an extension of the notion of  probability and the Lebesgue integral to the non-additive case, while preserving the monotonicity. 
Capacities (in the modern terminology) often (but not always) preserve the monotone sequential continuity as well.
These tools have attracted considerable interest in various fields: in the probability and stochastic analysis community (cf., e.g., \cite{SPS_1971__5__77_0}), 
in fuzzy theory (cf., e.g., \cite{denneberg1994non}, \cite{Grab}, \cite{Grab1}), 
in finance (cf., e.g., \cite{cherny2009new}, \cite{chateauneuf1994modeling})
in actuarial science  and in economics (cf., e.g., \cite{gilboa1987expected}, \cite{schmeidler1989subjective}, \cite{chateauneuf2000optimal}, \cite{FöllmerSchied+2016}), in machine learning (cf., e.g., \cite{ML}), etc.
A particular case of Choquet capacity, known as distorted probability, and the corresponding Choquet integral, known as distorted integral, hold a prominent place in decision theory as alternatives to the classical expected utility of Von Neumann and Morgenstern (cf., e.g., \cite{yaari1987dual}, \cite{QUIGGIN1982323}).
At the interface of mathematical finance and actuarial mathematics (possibly, up to a minus sign, depending on the sign convention), the Choquet integrals with respect to a distorted probability are known as (static) distortion risk measures or distortion premium principles.
It is by now well-known that a number of risk measures used extensively by practitioners (and recommended by regulatory bodies) are of this form, such as the Value at Risk and the Average Value at Risk. 
Distorted probabilities and distorted integrals have recently attracted renewed interest, especially in the conditional probability framework: 
cf., e.g., \cite{vega2021conditional} (for convex conditional risk measures),
\cite{de2023conditional} (for conditional quantiles), 
\cite{zang2024random} (for random distortion risk measures), 
and \cite{de2023convex} (for further actuarial applications). 

The field of stochastic orderings induced by or related to non-linear integrals has also been an active field of research in the recent years with various applications: cf., e.g., \cite{grigorova2014stochastic}, \cite{grigorova2014risk} and \cite{GRIGOROVA201373} for stochastic orderings with respect to a capacity, 
\cite{SelLy2022NicolasPrivault} and \cite{SelLy2021NicolasStochasticordering} for stochastic orderings induced by non-linear $g$-evaluations and by non-linear (sub-additive) $G$-expectations, \cite{song2006representations} and \cite{song2009risk} for distortion risk measures (and suprema thereof) with respect to a probability and classical stochastic dominance relations.

In the present paper, we place ourselves in the capacity framework where $(\Omega, \F, c)$ is an underlying capacity space. 
In  economic, financial, and actuarial applications, the capacity $c$ typically captures ambiguity on the model (where there might not even exist a reference probability or a reference set of probabilities). 
We study monotone comonotonic additive risk operators $\rho$ from $\chi(\F)$ to $\chi(\G)$, where $\G$ is a sub-$\sigma$-algebra of $\F$ and $\chi(\F)$ (resp. $\chi(\G)$) denotes the space of bounded $\F$-(resp. $\G$-)measurable functions. 
Our operators are also monotone for  the increasing or for the stop-loss stochastic ordering relations with respect to the capacity $c$.
We extend the (static) Choquet integrals with respect to a distorted capacity to the conditional setting, by "randomising" the distortion function which now depends also on $\omega \in \Omega$.  
Such random distortion functions (possible with slight variants in the definition) have appeared in the work  \cite{cherny2009new}, and, more recently,  in \cite{vega2021conditional} and \cite{zang2024random} for the case where $c$ is a probability measure.
We study the properties of the $\G$-randomly distorted Choquet integrals as operators from $\chi(\F)$ to $\chi(\G)$ and we show that they are comonotonic additive. 
In the case of a (general) $\G$-random distortion function $\phi^{\G}$, we characterise these integrals in terms of comonotonic additive conditional risk measures which are monotone with respect to the first-order stochastic dominance (with respect to the capacity $c$).
In the case where the random distortion function $\phi^{\G}$ is moreover concave, we characterise them in terms of comonotonic additive conditional risk measures which are monotone with respect to the stop-loss stochastic dominance (with respect to the capacity $c$).
For related results in the case where $c$ is a probability measure, we refer to \cite{vega2021conditional}.
We provide examples, inspired from actuarial science and finance, which can be seen as  \textit{randomised} versions of well-known risk measures such as the  Value at Risk (VaR) and the Average Value at Risk (AVaR) in the probabilistic or, more generally, in the capacity framework. 

The remainder of the paper is organised as follows: 
In Section \ref{sec_Preliminaries}, we provide some preliminaries on the notion of  capacity, stochastic orderings with respect to a capacity, and quantile functions with respect to a capacity, which are essential for the sequel. 
In Section \ref{sec_rand_dis_Choq}, we introduce the Choquet integrals with respect to a $\G$-randomly distorted capacity and show their main properties, among which the comonotonic additivity, under  suitable assumptions on the capacity space. 
Section \ref{sec_representation_first_order} is dedicated to our first representation result for comonotonic additive conditional risk measures which are monotone with respect to the first-order stochastic dominance (with respect to the capacity $c$). 
Section \ref{sec_representation_stop_loss} is devoted to the case where the random distortion $\phi^{\G}$ is concave and to our second representation result for comonotonic additive conditional risk measures which are monotone with respect to the stop-loss stochastic ordering (with respect to the capacity $c$).
Section \ref{sec_Examples} contains the examples in the probabilistic framework and in the capacity framework. The Appendix \ref{sec_Appendix} contains  the  proofs of some technical lemmas.

\section{Preliminaries}\label{sec_Preliminaries}
 
    Throughout this work, $ \left(\Omega, \F \right) $ will be a given measurable space and $\G \subseteq \F$ will be a sub-$\sigma$-algebra. 
    We denote by  $\chi(\F)$ the set of all bounded, real-valued, $\F$-measurable functions on  $\Omega$. 
 Similarly, we denote by $\chi(\G)$
 the set of all bounded, real-valued, $\G$-measurable functions on  $\Omega$. The supremum norm is defined by $||X||:=\sup_{\omega \in \Omega}|X(\omega)|$.
 For $X$ and $Y$ in $\chi$, we write $X=Y$ (resp. $X\leq Y$), when    
 $X(\omega)=Y(\omega)$ (resp. $X(\omega)\leq Y(\omega)$), for all $\omega\in\Omega$.

\subsection{Conditional risk measure and comonotonic additivity}
We start with a general  definition of conditional risk measure, which takes as input an element of $\chi(\F)$, interpreted as a random loss, and returns an element in $\chi(\G)$. 
 \begin{definition}[Conditional risk measure] 
   Any operator  $\rho:\chi(\F) \rightarrow  \chi(\G)$ is called a $\G$-conditional risk measure or just conditional risk measure (if there is no confusion about the $\sigma$-algebra $\G$ to which we refer).
\end{definition}
Let us recall the definition of comonotonic measurable functions, especially popular in the risk measurement literature. 

\begin{definition} \label{defn4.2}
   Let  $X$ and $Y$ be two real-valued measurable functions on $ \left(\Omega,\mathcal{F}\right)$. We say that $X$ and $Y$ are comonotonic 
   if for any pair $( \omega , \Bar{\omega} ) \in \Omega \times \Omega$,
     $   \left( X(\omega ) - X(\Bar{\omega} ) \right) 
        \left(Y(\omega ) - Y(\Bar{\omega} ) \right)  
    \geq 0$.
\end{definition}
\begin{definition}[Comonotonic additivity] \label{defn1}
     Let $\rho: \chi(\F) \rightarrow  \chi(\G)$. We say that $\rho$ is comonotonic additive if $\rho(X + Y)(\omega) =  \rho(X)(\omega) + \rho(Y)(\omega) $, for all $\omega \in \Omega$,
        for all $X\in \chi(\F)$ and $Y \in \chi(\F)$ which are comonotonic. 
\end{definition}
\begin{remark}\label{remark_30}
If $\rho$ is comonotonic additive, 
Lipschitz continuous (with respect to the supremum norm $||\cdot||$) and normalised (that is, $\rho(1)=1$), then $\rho$ satisfies the positive homogeneity property (that is, for any $a\geq 0$, for any $X\in\chi(\F)$, for any $\omega\in\Omega$,  $\rho(aX) (\omega)=a\rho(X)(\omega)$). Indeed, let $X \in \chi(\F)$. We can show by induction that, due to the comonotonic additivity of $\rho$, for any integer $n\geq 1$,
$ \rho(nX)(\omega)=  
n\rho(X)(\omega)$. 
    Let $\beta$ be in $\mathbb{Q} \cap [0,\infty)$, then $\beta$ can be written as $\beta = \frac{n}{m}$, where $n$ and  $m$ are positive integers. By using the property for positive integers, we have: 
 $ n\rho(X)(\omega)= \rho(nX)(\omega)=   \rho(m \frac{n}{m} X)(\omega)= m  \rho(\frac{n}{m}X)(\omega) $,
   which gives $ \rho(\frac{n}{m} X)(\omega)=  \frac{n}{m} \rho(X)(\omega)$.
     For any $a \geq 0$, there exists a sequence of rational numbers $\beta_n  \in \mathbb{Q} \cap [0,\infty)$ such that $\lim_{n \rightarrow \infty} \beta_n = a$.
    We have: $\lim_{n \rightarrow \infty} \beta_nX(\omega) = aX(\omega)$ for all $\omega \in \Omega$, and  the convergence is also in sup-norm as $ 
     \lim_{n \rightarrow \infty} ||(\beta_n - a)X||= \lim_{n \rightarrow \infty} |\beta_n - a| \, ||X|| =0$. 
    As $\rho$ is Lipschitz continuous, we have:  
     $ \rho(aX)  = \lim_{n \rightarrow \infty} \rho(\beta_n X)$ in sup-norm,
 which implies the pointwise convergence 
 $ \rho(aX)(\omega)  = \lim_{n \rightarrow \infty} \rho(\beta_n X) (\omega)$. 
  Hence,
    \begin{equation*}
  \rho(aX) (\omega)
  = \lim_{n \rightarrow \infty} \rho(\beta_n X) (\omega)
   =   \lim_{n \rightarrow \infty} \beta_n \rho(X) (\omega)
  = a \rho (X)(\omega),   \,\, \text{for all} \,\,  \omega \in \Omega.
    \end{equation*}
\end{remark}

\subsection{Capacities, quantile functions with respect to capacities, and stochastic orderings}
In this subsection, we recall some definitions and properties of capacities, which can be found in \cite{denneberg1994non} and/or \cite{FöllmerSchied+2016}. 
 Moreover, we recall some results related to the first-order stochastic dominance with respect to a capacity, which is an extension of the classical first-order stochastic dominance (with respect to a probability measure). For more results on stochastic orderings with respect to a capacity, we refer to 
 \cite{grigorova2014risk} and \cite{grigorova2014stochastic}.

 \begin{definition}
 Let $ (\Omega,\F)$ be a measurable space.
     A set function $ c: \F \rightarrow [0,1] $ is called \emph{a capacity} if it satisfies $c \left( \emptyset \right ) = 0 $   (groundedness), $c \left( \Omega \right ) = 1 $ (normalisation) and the 
     monotonicity property: if $A,B \in \F$ are such that $A \subset B $, then $ c(A) \leq  c (B)$.
 \end{definition}

 \begin{definition}  \label{def_continuous_B_C}
 A capacity $c$ is called\emph{ continuous from below} if for any sequence of events 
   $ (A_n) \subset \F^\mathbb N$  such that $A_n \subset A_{n+1}$, for all $ n \in \mathbb{N}$, we have  $\lim_{n \rightarrow \infty} c(A_n) = c(\cup_{n=1}^{\infty} A_n)$.
\end{definition}

\begin{remark}Let $\mathcal{P}$ be a non-empty set of probability measures. 
    Then, $c(\cdot):=\sup_{P \in \mathcal{P}} P (\cdot)$ is a capacity which is continuous from below. Note, however, that the non-linear functional $\sup_{P \in \mathcal{P}} \EE_P(\cdot)$ is in general different from the Choquet integral $\EE_c(\cdot)$. 
    
    \end{remark}

 We recall the notion of (non-decreasing) distribution function with respect to a capacity (cf., e.g., Definition 4.85 in \cite{FöllmerSchied+2016}).
\begin{definition}  \label{def_CDF_G}
  Let $c$ be a given capacity.  
  For a real-valued measurable function $X$, \emph{the distribution function} $G_X$ \textit{of} $X$\textit{ with respect to} $c$ is defined by
 $ G_{X,c}\left( x\right) = G_X\left( x\right) := 1 - c \left( X > x \right)$, for all $x \in \mathbb{R}$.
\end{definition}

We extend the above definition by setting $G_X(+\infty) := 1$  and $G_X(-\infty) := 0$.

The following assumption is imposed in some results in the sequel.

\begin{assumption} \label{assumption_Uniform_capacity}
    Let $c$ be a capacity. 
    We assume that there exists a real-valued measurable function $Z$ with distribution function with respect to the capacity $c$, $G_{Z}$, which is continuous and satisfies: 
\[
 \lim_{x\rightarrow -\infty} G_{Z} (x) =0  \,\,\,\,   and \,\,\,\,  \lim_{x\rightarrow +\infty} G_{Z} (x) =1.
\]
\end{assumption}

The following lemma can be found in \cite{grigorova2014risk}.
It is well-known in the case where $c$ is a probability measure as a way of constructing a random variable with a uniform distribution on $[0,1]$.

\begin{lemma} \label{lemma_38}
    We assume Assumption \ref{assumption_Uniform_capacity}.
Set $U:=G_{Z}(Z)$.
Then, the distribution function $G_{U}$ of $U$ is:
$G_{U}(x)=\left\{ \begin{array}{cc}
    0,  &  \text{if}\,\, x<0,\\
    x,  &  \text{if} \,\, x \in [0,1],\\
    1, &  \text{if}\,\, x>1.
 \end{array} \right.$
\end{lemma}  

\begin{remark}
    
   Let us consider the particular case  where $c=P$ is a probability measure. In this case,  Assumption~\ref{assumption_Uniform_capacity} amounts to $(\Omega, \F, P)$ being atomless (which can be seen as  a standard assumption in the risk measurement literature), and which is equivalent to $(\Omega, \F, P)$ carrying a uniform  random variable on the interval $[0,1]$. 
\end{remark}

For any $X\in\chi(\F)$, the function $G_X$ is non-decreasing (as $c$ is monotone); it is thus possible to define its generalised inverse (cf., e.g., the Appendix of \cite{FöllmerSchied+2016}).

\begin{definition}
 Let $c$ be a capacity.  
  For a real-valued measurable function $X$, 
every function $ r_X : \left( 0 , 1 \right) \rightarrow \mathbb{R} $ satisfying 
    \[
    \sup \{ x \in \mathbb{R}: G_{X}( x) < t  \}
    \leq
    r_X \left( t\right) 
    \leq
     \sup \{ x \in \mathbb{R} : G_{X}( x) \leq t  \}, \,\, \,\, for \,\, all \,\, t \in  \left( 0 , 1 \right)
    \]
    is called \emph{a quantile function} of $X$ with respect to $c$, where the convention $\sup\{\emptyset\} = -\infty$ is used.
    \\
The functions $r_{X}^-$ and $r_{X}^+$ defined by 
\[
 r_{X}^-(t) :=\sup \{x \in \mathbb{R}: G_X( x) < t  \}  \,\, \text{for all} \,\, t \in (0,1),
\]
\[
r_{X}^+ (t):= \sup \{ x \in \mathbb{R}: G_X( x) \leq t  \} \,\, \text{for all} \,\, t \in (0,1)
\]
 are called \emph{the lower }and \emph{upper quantile functions} of $X$ with respect to $c$.
\end{definition}
We extend the definition to $t=0$ and $t=1$, by setting
$r_{X}^+(1):= r_{X}^-(1):= \lim_{t \uparrow 1}r_{X}^-(t)=\sup_{t < 1}r_{X}^-(t)$ 
and 
$r_{X}^-(0):= r_{X}^+(0):= \lim_{t \downarrow 0}r_{X}^+(t) =\inf_{t>0} r_{X}^+(t)$. 

 For notational convenience, we omit the dependence on $c$ in the notation $G_X$ and
$r_X$ when there is no possible confusion about the capacity.
 The following remark (cf. \cite{FöllmerSchied+2016}) gives us another useful way of writing the lower and upper quantile functions (with respect to a capacity).  
\begin{remark} \label{remark_44}
    The lower and upper quantile functions of $X$ with respect to $c$ can be written: for all $t \in (0, 1)$,
    \[
 r_{X}^-(t) =\inf \{x \in \R: G_X( x) \geq t  \}, \,\,\,\,\,\,\,\,\,
r_{X}^+ (t)= \inf \{ x \in \R: G_X( x) > t\}.
\]
\end{remark}

 The following observation on the distribution function (with respect to a capacity)  of a bounded measurable function $X$ can be found in \cite{grigorova2014risk} (cf. Remark 2.5). 
 \begin{remark}\label{remark_45}  
 Let $c$ be a capacity. For $X \in \chi(\F)$, we have:\\ $\lim_{x \rightarrow -\infty}G_X( x)=0$ and $\lim_{x \rightarrow +\infty}G_X( x)=1$.
 We denote by $G_X(x-)$ and $G_X(x+)$ the left-hand and right-hand limits of $G_X$ at $x$ (these limits exist as $G_X$ is non-decreasing).
 A function $r_X$ is a quantile function of $X$ (with respect to $c$) if and only if $ G_X(r_X(t)-) \leq t \leq G_X(r_X(t)+), \,\,\text{for all} \,\, t \in (0,1).$
     \end{remark}

To fix the notation, we recall that the (classical) Choquet 
     integral of  $X \in \chi(\F)$ with respect to a capacity $c$  
     is defined by: 
    \[
   \EE_{ c}(X) = \int_0^{+\infty}  c(X > x) dx
    +  \int_{- \infty}^{0} \left[ c(X > x) -1 \right] dx.
    \] 

We also recall the definition of the first-order stochastic dominance with respect to a capacity (cf. \cite{grigorova:tel-00878599} for more results on this ordering relation), generalizing the classic first-order stochastic dominance (with respect to a probability).\\
  \begin{definition}
Let $X$ and $Y$ be in $\chi(\F)$ and let $c$ be a capacity. We say that $X$ is smaller than $Y$ in the first-order stochastic dominance (or in the first-order stochastic ordering) with respect to the capacity $c$, denoted by $ X \preceq_{st,c} Y $, if:
\begin{equation}
      \mathbb{E}_c \left( u\left( X \right ) \right ) \leq \mathbb{E}_c \left( u\left( Y \right )  \right ) 
\end{equation}
for all deterministic functions $u: \mathbb{R} \rightarrow \mathbb{R} $ which are non-decreasing.
\end{definition}
From an economic point of view, $X$ is smaller than $Y$ in the first-order stochastic dominance  with respect to the capacity $c$, if all Choquet expected utility (CEU) maximizers (having a non-decreasing utility function) prefer having $Y$ to $X$.  

 The following lemma gives some  useful tools when working with  the first-order stochastic dominance with respect to the capacity $c$. 

 \begin{lemma}\label{charc_1_c}  
Let $X$ and $Y$ be in $\chi(\F)$. 
  Let $c$ be a capacity. 
  If $ X \preceq_{st,c} Y $, then all of the following statements hold true:
 \begin{enumerate}
  \item[(i)]  $ G_X ( x ) \geq  G_Y ( x)$ for all $x \in \mathbb{R} $.
   \item[(ii)] $ r^+_X(t) \leq  r^+_Y (t)$ for all $t \in [0, 1]$.
    \item[(iii)] $ r^-_X (t) \leq  r^-_Y (t)$ for all $t \in [0, 1]$. 
     \end{enumerate}
  \end{lemma}
  \begin{proof}
      For the proof of the first two statements, we refer to the proof of Proposition 3.3.1 in \cite{grigorova:tel-00878599}.
   The proof of the third statement is similar.     We provide a full proof of the lemma  in the appendix for the convenience of the reader.\qed
  \end{proof}

In the following, we define the "distribution invariance" property of $\rho$ (with respect to $c$).
It is used
in \cite{grigorova2014risk} (cf. Remark 3.5) for the case where $\rho$ is a static risk measure on a capacity space. 

 \begin{definition}["Distribution invariance"]\label{defn5}
 Let $c$ be a capacity. 
 Let $\rho: \chi(\F) \rightarrow \chi(\G)$ .
 We say that $\rho$ is "distribution invariant" (with respect to $c$) if:  
 $G_X(x)=G_Y(x)$, for all $ x \in \R$, implies $\rho(X)(\omega)=\rho(Y)(\omega), $ for all $\omega \in\Omega$. 
 \end{definition}
 In other words, $\rho$ is "distribution invariant" (with respect to $c$) if any two random losses $X$ and $Y$ with the same distribution function (with respect to $c$) have also the same risk ($\rho(X)(\omega)=\rho(Y)(\omega), $ for all $\omega \in\Omega$). 
\section{The randomly distorted Choquet integral with respect to a $\G$-randomly distorted capacity} \label{sec_rand_dis_Choq}
In this section, we consider $\G$-randomly distorted Choquet integrals with respect to a $\G$-randomly distorted capacity $\phi^{\G} \circ c$, where $\phi^{\G}$ is a $\G$-measurable random distortion function (which we call also $\G$-random distortion function).   
  The following is the precise definition of a $\G$-random distortion function (cf. also \cite{vega2021conditional} and \cite{de2023convex}).
\begin{definition} 
   We say that a mapping $ \phi^{\G}: \Omega \times [0,1] \rightarrow [0,1]$ is a $\G$-random distortion function (or random distortion function, if 
   no 
   possible confusion  on the $\sigma$-algebra) if:
   \begin{itemize}
      \item[1)] For any $\omega \in \Omega$, the function $t\in [0,1] \mapsto \phi^{\G}(\omega, t)$ is non-decreasing and normalised, that is, $\phi^{\G}(\omega, 0)=0$ and  $\phi^{\G}(\omega, 1)=1$.
      \item[2)] For any $t \in [0,1]$, $\omega \mapsto \phi^{\G}(\omega, t)$ is  $\G$-measurable. 
   \end{itemize}
\end{definition}

When $\G=\{\Omega, \emptyset\}$, then $\phi^{\G}=\phi$ is deterministic (does not depend on $\omega$). 
In this case, we recover the usual definition of a (deterministic) distortion function, popular in decision theory and in risk measures.  

We define the randomly distorted Choquet integral (with respect to a $\G$-random distortion function  $\phi^{\G}$ and a capacity $c$) as follows:

\begin{definition} \label{defn3}
Let $c$ be a capacity on $ (\Omega,\F)$. 
Let $\phi^{\G}$ be a $\G$-random distortion function.
    The $\G$\emph{-randomly distorted Choquet}
    \emph{integral} of a measurable function $X \in \chi(\F)$ 
     is defined by: for all $\omega \in \Omega$,
    \[
   \EE_{\phi^{\G} \circ c}(X)(\omega) = \int_0^{+\infty} \phi^{\G}(\omega, c(X > x)) dx
    +  \int_{- \infty}^{0} \left[\phi^{\G}(\omega, c(X > x)) -1 \right] dx.
    \]
\end{definition} 
   Since 
   $\phi^\G$ is  a $\G$-random distortion, 
$ \EE_{\phi^{\G} \circ c}(X)$ is $\G$-measurable. 
Let us note also that when $\phi^{\G}=\phi$ is deterministic and $c$ is a probability measure, 
we recover the usual Choquet integral with respect to a distorted probability.
When $\phi^{\G}=\phi=id$ is the deterministic identity function  and $c$ is a capacity, we recover the usual Choquet integral with respect to the capacity $c$.  

\begin{proposition} \label{prop_38}
The randomly distorted Choquet integral $\EE_{\phi^{\G} \circ c}$ satisfies the following properties:
\begin{itemize}
\item[(i)] ("Distribution invariance") If $X, Y \in \chi(\F) $ are such that $G_X(x) = G_Y(x)$ for all $x \in \R$, then $\EE_{\phi^{\G} \circ c}(X) =\EE_{\phi^{\G} \circ c}(Y)$.
        \item [(ii)] (Monotonicity) If $X, Y \in \chi(\F)$ are such that $ X \leq Y $, then $  \EE_{\phi^{\G} \circ c}(X) \leq \EE_{\phi^{\G} \circ c}(Y)$.
       \item [(iii)] (Translation invariance)  For $X \in \chi(\F)$, for $a \in \R$, we have: $  \EE_{\phi^{\G} \circ c}(a+X) = a + \EE_{\phi^{\G} \circ c}(X)$.
       \end{itemize}
\end{proposition}
\begin{proof}
\begin{itemize}
\item[(i)] It follows from Definition \ref{defn3} and from the definition of $G_X$ and $G_Y$.
    \item[(ii)] 
     Let $X, Y $ be in $\chi(\F)$ such that $X \leq Y$. By the monotonicity of $c$, for any given $x \in \R$, we have:
   $c(X > x)  \leq c(Y > x)$. 
   Therefore, $\phi^{\G} \left( \omega, c(X > x) \right) \leq \phi^{\G} \left( \omega, c(Y > x) \right)$,  for each $\omega \in \Omega$ (by using the monotonicity of $\phi^{\G}$).
   Hence, we obtain the desired result.
   \item[(iii)] Let $a \in \R$. By Definition \ref{defn3}, we have: for all $\omega \in \Omega$, 
      \[\begin{split}
     &\EE_{\phi^{\G} \circ c}(X+a)(\omega) 
     \\&=   \int_{0}^{+\infty} \phi^{\G} \left( \omega, c(X > x-a) \right)dx+ \int_{-\infty}^{0} [\phi^{\G} \left( \omega, c(X > x-a) \right) -1] dx
       \\&=
        \int_{-a}^{+\infty} \phi^{\G} \left( \omega, c(X > t) \right)dt+ \int_{-\infty}^{-a} [\phi^{\G} \left( \omega, c(X > t) \right) -1] dt
             \\&=
           a+  \int_{0}^{+\infty} \phi^{\G} \left( \omega, c(X > t) \right)dt + \int_{-\infty}^{0} [\phi^{\G} \left( \omega, c(X > t) \right) -1] dt 
       \\& = a+\EE_{\phi^{\G} \circ c}(X)(\omega),
       \end{split}\]
       where we have used a change of variables.\qed
\end{itemize}    
\end{proof}
The translation invariance is sometimes referred to as cash invariance in the risk measurement literature. 
Our next important goal is to establish the comonotonic additivity of $\EE_{\phi^{\G} \circ c}(\cdot)$.
\begin{theorem}[Comonotonic additivity]\label{Theorem_15} 
  Let $c$ be a continuous from below capacity. 
We assume Assumption \ref{assumption_Uniform_capacity}.
Let $\phi^{\G}$ be a $\G$-random distortion function.
 If $X$ and $Y$ in $\chi(\F)$ are comonotonic,
  then:  \begin{equation} \label{eq156}
  \EE_{\phi^{\G} \circ c}(X+Y) = \EE_{\phi^{\G} \circ c}(X) + \EE_{\phi^{\G} \circ c}(Y). 
  \end{equation}
\end{theorem}

We present some results which will be used in the proof of Theorem \ref{Theorem_15}, and which are interesting on their own account.

Let $X$ be a non-negative step function defined by:
\begin{equation} \label{eq154}
   X = \sum_{i=1}^n x_i\1_{A_i}, 
\end{equation} 
where the $x_i$'s are non-negative real numbers, and where the $ A_i$'s are disjoint and form a measurable partition of $\Omega$.  Without loss of generality, we can assume that the numbers $x_i$ are ranged in descending order (i.e. $x_1 \geq x_2 \geq \cdots \geq x_n \geq 0$).

\begin{lemma}[The $\G$-randomly distorted Choquet integral formula for a non-negative step function] \label{lemma_36}
Let $c$ be a capacity on $ (\Omega,\F)$. 
For a non-negative step function $X = \sum_{i=1}^n x_i\1_{A_i}$ of the form defined in \eqref{eq154},
we have: for all $\omega \in \Omega$,
\begin{equation}\label{eq153}
 \EE_{\phi^{\G} \circ c}\left(\sum_{i=1}^n x_i\1_{A_i}\right)(\omega) = \sum_{i=1}^n  (x_i-x_{i+1}) \, \phi^{\G}(\omega, c(\cup_{k=1}^i A_k)), \,\,  \text{where} \,\,  x_{n+1} := 0.
\end{equation} 
\end{lemma}
\begin{proof}
We have: for all $\omega \in \Omega$,
\[ \begin{split}
\EE_{\phi^{\G} \circ c}(X)(\omega) &= \int_0^{+\infty} \phi^{\G}(\omega , c(X > z)) dz  = \int_0^{x_1} \phi^{\G}(\omega , c(X > z)) dz
\\ &=  \sum_{i=1}^n \int_{x_{i+1}}^{x_i} \phi^{\G}(\omega, c(X > z)) dz
=  \sum_{i=1}^n \int_{x_{i+1}}^{x_i} \phi^{\G}(\omega, c(\cup_{k=1}^i A_k)) dz
\\ &=  \sum_{i=1}^n \phi^{\G}(\omega, c(\cup_{k=1}^i A_k)) \int_{x_{i+1}}^{x_i} dz
=  \sum_{i=1}^n  (x_i-x_{i+1}) \, \phi^{\G}(\omega, c(\cup_{k=1}^i A_k)),
\end{split}
\]
 where we have used that $\{X > z \}=\cup_{k=1}^i A_k$, for each $z \in [x_{i+1}, x_i)$ (since $x_k > z$ for $k \leq i $).\qed
\end{proof}

For the case where $\phi^{\G}$ is deterministic, the above formula \eqref{eq153} can be found, for instance, in \cite{schmeidler1986integral}.

In the following lemma, we show that the $\G$-randomly distorted Choquet integral is Lipschitz continuous on $\chi(\F)$ endowed with the sup-norm $||\cdot||$.

\begin{lemma}[Lipschitz continuity of $\G$-randomly distorted Choquet integrals]\label{lemma_37} 
Let $c$ be a capacity. 
The $\G$-randomly distorted Choquet integral 
is Lipschitz continuous (with respect to the supremum norm $||\cdot||$);
more precisely,
for all $X, Y$ in $\chi(\F)$, we have: 
\[
||\EE_{\phi^{\G} \circ c}(X) - \EE_{\phi^{\G} \circ c}(Y)|| \leq ||X-Y||.
\]
\end{lemma}

\begin{proof}
The proof of the lemma uses the same arguments as Lemma 4.3 in \cite{FöllmerSchied+2016}. 
Let $X, Y \in \chi(\F)$. Then, $ X \leq Y + ||X-Y||$ 
and  $ Y \leq X + ||X-Y||$. 
 Therefore, by the monotonicity and the cash invariance of $\EE_{\phi^{\G} \circ c}$ we have:  for all $\omega \in \Omega$,
 \[ \begin{split}
 \EE_{\phi^{\G} \circ c}(X)(\omega) \leq \EE_{\phi^{\G} \circ c}(Y+ ||X-Y||) (\omega) =\EE_{\phi^{\G} \circ c}(Y)(\omega) + ||X-Y|| 
 \\ 
 \EE_{\phi^{\G} \circ c}(Y)(\omega) \leq \EE_{\phi^{\G} \circ c}(X+ ||X-Y||) (\omega)  =\EE_{\phi^{\G} \circ c}(X) (\omega)+ ||X-Y||.
 \end{split} \] 
 Hence, we obtain: $|\EE_{\phi^{\G} \circ c}(X)(\omega) - \EE_{\phi^{\G} \circ c}(Y)(\omega)| \leq ||X-Y||$ for all $\omega \in \Omega$.
  Hence, $||\EE_{\phi^{\G} \circ c}(X) - \EE_{\phi^{\G} \circ c}(Y)|| \leq ||X-Y||$.\qed
\end{proof}
 \begin{remark}\label{remark_32}
     In fact, this result holds true for any $\rho: \chi(\F) \rightarrow \chi(\G)$ such that $\rho$ is monotone and cash invariant. This result is well-known in the (conditional) risk measurement literature (cf., e.g., \cite{FöllmerSchied+2016}).
 \end{remark}
 \begin{remark} \label{remark_40}
 As a consequence of Lemma \ref{lemma_37}, $\EE_{\phi^{\G} \circ c}$ is a continuous (non-linear) operator in sup-norm. 
  If $ ||X_n-X|| \rightarrow 0$ as $n \rightarrow \infty$, 
    then $ || \EE_{\phi^{\G} \circ c}(X_n) - \EE_{\phi^{\G} \circ c}(X)|| \rightarrow 0$ as $n \rightarrow \infty$.
\end{remark} 

The following lemma provides a useful characterisation of the comonotonicity of two non-negative step functions (cf. \cite{schmeidler1986integral}).  We provide in the Appendix a proof of this lemma for the convenience of the reader.

\begin{lemma} \label{lemma_30}
    Two non-negative step functions $X$ and $Y$ in $\chi(\F)$ are comonotonic if and only if there exists a positive integer $n$, a partition of $\Omega$ into $n$ pairwise disjoint sets in $\F$, $(A_i)_{i=1}^n$, and real numbers $x_1 \geq x_2 \geq \cdots \geq x_n \geq 0$ and $y_1 \geq y_2 \geq \cdots \geq y_n \geq 0$
    such that $X= \sum_{i=1}^n x_i\1_{A_i}$ and $Y= \sum_{i=1}^n y_i\1_{A_i}$. 
\end{lemma}

\begin{proposition}[Approximation of non-negative comonotonic pair in $\chi(\F)$] \label{prop_39} 
     Let $c$ be a continuous from below capacity. 
We assume Assumption \ref{assumption_Uniform_capacity}.
 Let $U:=G_Z(Z)$ (as in Lemma \ref{lemma_38}).
   Let $X, Y$ be two non-negative \textbf{comonotonic} measurable functions in $ \chi(\F)$.
     Then, we have the following:
     \begin{itemize}
     \item $X$ and $ r_X^-(U)$ (respectively $Y$ and $r_Y^-(U)$) have the same distribution function with respect to $c$.
    \item  Moreover, $X+Y$ and $r^-_X(U)+r^-_Y(U)$ have the same distribution function with respect to $c$.
       \item There exist two sequences of non-negative step functions $(X_n)$ and $(Y_n)$ such that $X_n$ and $Y_n$ are comonotonic, for each $n$, and 
    such that $ ||X_n-r_X^-(U)|| \rightarrow 0$, $ ||Y_n-r_Y^-(U)|| \rightarrow 0$ and $ ||X_n+Y_n - (r_X^-(U)+r_Y^-(U))|| \rightarrow 0$ as $n \rightarrow \infty$.
     \item Moreover, for all $\omega \in \Omega$,
     $ \EE_{\phi^{\G} \circ c} (X_n)(\omega) \xrightarrow[n \rightarrow \infty]{}    \EE_{\phi^{\G} \circ c}(X)(\omega) $,
            $\EE_{\phi^{\G} \circ c} (Y_n)(\omega)  \xrightarrow[n \rightarrow \infty]{}    \EE_{\phi^{\G} \circ c}(Y)(\omega) $ 
            and 
            $ \EE_{\phi^{\G} \circ c} (X_n+Y_n)(\omega)  \xrightarrow[n \rightarrow \infty]{}    \EE_{\phi^{\G} \circ c}(X+Y)(\omega) $.
     \end{itemize}
\end{proposition}

In order to prove the above proposition, we make the following observations
(for the case where $c$ is a probability measure, they can be found, for instance, in \cite{embrechts2013note} 
or \cite{FöllmerSchied+2016}).

\begin{remark}\label{remark_41}
Let $c$ be a capacity. 
If $c$ is \emph{continuous from below}, then $G_X$ is right-continuous and $r_X^-$ is left-continuous. 
Indeed, let $(x_n)$ and $x$ be such that $x_n \downarrow x$.
We set $A_n:=\{X>x_n\}$ for each $n \in \mathbb{N}$, and $A:=\{X>x\}$. 
Then, $A_n \subset A_{n+1}$ for each $n \in \mathbb{N}$.
Moreover, $\cup_{n} A_n =A$.
Hence, as $c$ is continuous from below, we get $c(A)=c(\cup_n A_n ) = \lim_{n \rightarrow \infty}c( A_n)$. 
Therefore, we obtain: $ \lim_{n \rightarrow \infty}G_X(x_n)=G_X(x)$.  
\end{remark}

 \begin{remark}\label{remark_39}
  Let $c$ be a capacity. 
  \begin{itemize}
      \item[a)] For any $X \in \chi(\F)$, 
    $G_X(x) \geq t$ implies $r^-_{X} (t) \leq x$.
  Indeed, by Remark \ref{remark_44}, 
  we have: $r_X^-(t)=\inf\{x \in \R :G_X(x) \geq t\}$ for $t\in (0,1)$.
   Hence, if $x \in \R$ is such that $G_X(x) \geq t$, then $r^-_{X} (t) \leq x$.

   \item[b)] Moreover, if $c$ is\textit{ continuous from below}, the other implication also holds true; more precisely, $r^-_{X} (t) \leq x$ implies $G_X(x) \geq t$.
     Indeed, as $c$ is continuous from below, $G_X$ is right-continuous (cf. Remark \ref{remark_41}). 
     As a consequence, we can show that $G_X(r^-_{X} (t)) \geq t$. 
   Indeed, let $A:=\{x\in \R:G_X(x) \geq t\}$. Then, by Remark \ref{remark_45}, there exists a sequence $(x_n) \subseteq A$ such that $x_n \downarrow \inf A = r^-_{X} (t)$ and $G_X(x_n) \geq t$. 
      Hence, by the the right-continuity of $G_X$, $t \leq \lim_{n \rightarrow \infty} G_X(x_n)= G_X(r^-_{X} (t))$, which gives $G_X(r^-_{X} (t)) 
      \geq t$.
      Therefore, if 
      $r^-_{X} (t) \leq x$, then by the non-decreasingness of $G_X$ and by the fact that $t \leq G_X(r^-_{X} (t))$  we have $t \leq G_X(r^-_{X} (t)) \leq G_X(x) $.
   \item[c)] If $c$ is\textit{ continuous from below}, then $r_X^-$ is left-continuous. 
Indeed, let $(t_n)$ and $t$ be such that $t_n \uparrow t$.
Then, as $r_X^-$ is non-decreasing, we have: $r_X^-(t_n)$ is non-decreasing and $\lim_{n \rightarrow \infty} r_X^-(t_n) \leq r_X^-(t)$.
\\ We set $z:=\lim_{n \rightarrow \infty} r_X^-(t_n)$. 
It remains to show that $z\geq r_X^-(t)$.
For each $n$, $z=\lim_{k \rightarrow \infty} r_X^-(t_k) \geq r_X^-(t_n)$. 
    Then, by the non-decreasingness of $G_X$ and by the fact that $ G_X(r^-_{X} (t_n)) \geq t_n $ (which follows from the right-continuity of $G_X$ as shown in b)), we have:
    $  G_X(z) \geq  G_X(r^-_{X} (t_n)) \geq t_n $. 
    Therefore, 
    $  G_X(z) \geq t $. 
    Hence, $z\geq r_X^-(t)$ (by Remark \ref{remark_44}).
  \end{itemize}
 \end{remark}

\begin{proof}[Proof of Proposition \ref{prop_39}]
  Let $U:=G_Z(Z)$ (as in Lemma \ref{lemma_38}).
   Let us show $G_X(x)=G_{r_X^-(U)}(x)$ for all $x \in \R$.
   As the capacity $c$ is continuous from below, we use Remark \ref{remark_39}. 
   Hence, we have:
\[
  G_{r_X^-(U)}(x) = 1 - c(r_X^-(U)>x) = 1 - c(U>G_X(x)) = G_U(G_X(x)).
   \]
   By using Lemma \ref{lemma_38}, we obtain $G_U(G_X(x)) = G_X(x)$.
   Hence,  $X$ and $ r_X^-(U)$ have the same distribution function (with respect to $c$).
   \\ 
    Let us show $G_{X+Y}(x)=G_{(r_X^-+r_Y^-)(U)}(x)$ for all $x \in \R$.
    We set $f:=r_X^-+r_Y^-$. 
    We have that $f$ is non-decreasing; 
    also $f$ and $G_U$ do not have common discontinuities (as $G_U$ is continuous by Lemma \ref{lemma_38}). 
    Therefore, (cf., e.g., the arguments of the proof of Proposition 3.2 in \cite{yan2010short}), we have: 
    \[
    G_{(r_X^-+r_Y^-)(U)}(z)=G_{f(U)}(z)= G_U \circ \check{f}(z), 
    \]
    where $\check{f}(z):=\sup \{y:f(y)\leq z\}$ for all $z \in \R$. 
    By Lemma \ref{lemma_38}, 
    we have: $G_U(\check{f}(z))=\check{f}(z)$.
    Let us compute 
  $ \check{f}(z) =\sup \{y:(r_X^-+r_Y^-)(y)\leq z\}$.
   As $X$ and $ Y$ are comonotonic measurable functions, it follows by Proposition 1.3.1  in \cite{grigorova:tel-00878599}  
   that $(r_X^-+r_Y^-)(y) =r_{X+Y}^-(y)$ for all $y \in (0,1)$. 
   Hence, we have
   \[ \begin{split}
       \check{f}(z) &=\sup \{y:(r_X^-+r_Y^-)(y)\leq z\}
       =\sup \{y:r_{X+Y}^-(y)\leq z\} 
        =\sup \{y: y\leq G_{X+Y}(z)\}
       \\ &= G_{X+Y}(z),
   \end{split}
   \]
   where we have used Remark \ref{remark_39} (and the continuity from below of the capacity $c$) in the last but one equality.
   Thus,  $X+Y$ and $ (r_X^-+r_Y^-)(U)$ have the same distribution function (with respect to the capacity).
    \\ 
   Let us construct a suitable sequence of non-negative step functions $(X_n)_{n \in \mathbb{N}}$ that converges to $r_X^-(U)$ in $||\cdot||$. 
   For each $n$, let us denote $\Pi_n = \{ b^n_n, \dots, b^n_{1}, b^n_0\}$ a partition of $[0,1]$ with $  0= b^n_{n} \leq  b^n_{n-1} \leq \dots \leq b^n_{1} \leq b^n_{0}=1 $, such that $\max_{1\leq i \leq n} (b^n_{i-1} - b^n_{i}) \rightarrow 0$ as $n\rightarrow \infty$. 
   Let $X_n$ be a step function defined by: $X_n := \sum_{i=1}^n x^n_i \1_{A^n_i} $, where $x^n_i := r_X^-(b^n_i)$
   and  
   \begin{equation}
    A^n_{n}=\{0 \leq U\leq b^n_{n-1} \} \,\,\,\, and \,\,\,\,   A^n_{i}=\{ b^n_{i} < U \leq b^n_{i-1}\} \,\,\,\, for \,\,\,\, i= 1,\dots ,n-1.
 \end{equation}
 For $\epsilon >0$, there exists $\delta=\delta(\epsilon) >0$ and a positive integer $N(\delta)$ such that for all $n \geq N(\delta)$, we have:  
  $b^n_{i-1} - b^n_{i} < \delta$, for $i \in \{1, \cdots,n\}$ 
(since $\max_{1\leq i \leq n} (b^n_{i-1} - b^n_{i}) \rightarrow 0$ as $n \rightarrow \infty$). 
Hence, we have: for all $n \geq N(\delta)$,  
 \[   b^n_{i-1} - \delta < b^n_{i} \leq b^n_{i-1}, \,\, \text{for} \,\,  i \in \{1, \cdots,n\}. \] 
 Let us fix $n \geq N(\delta)$.
  Then, by the monotonicity and the left-continuity properties of $r^-_X(\cdot)$ (cf. Remark \ref{remark_39}), we obtain: for all $i \in \{1, \cdots,n\}$,
 $  
  r^-_X(b^n_{i-1}) -r^-_X(b^n_{i}) < \epsilon. 
  $
   Hence, 
   $\epsilon_n :=\max_{1\leq i \leq n}( r_X^-( b^n_{i-1})-  r_X^-(b^n_{i})) < \epsilon.$
  \\Suppose that $\omega \in A^n_{i}$, 
   that is, $U(\omega) \in (b^n_{i}, b^n_{i-1}]$. 
    Then, as $r_X^-$ is non-decreasing, we have: 
   \[
   r_X^-(b^n_{i}) \leq r_X^-(U(\omega))  \leq  r_X^-( b^n_{i-1}).
   \]
   Hence, we obtain, (for $\omega \in A^n_{i}$):
   \[
   |X_n(\omega)-r_X^-(U(\omega))|= |  r_X^-(b^n_{i}) - r_X^-(U(\omega)) | \leq 
   r_X^-( b^n_{i-1})-  r_X^-(b^n_{i})  \leq \epsilon_n < \epsilon. 
   \]
  Then, for all $\omega \in \Omega$, we get:
  $
  |X_n(\omega)-r_X^-(U(\omega))| \leq \epsilon_n < \epsilon.
  $
  Hence, $
  ||X_n-r_X^-(U)|| < \epsilon.
  $
 Hence, $ \lim_{n \rightarrow \infty} 
 ||X_n-r_X^-(U)|| =0$.
 \\   For the measurable function $r_Y^-(U)$, by using the same techniques (with $r_Y^-$ instead of $r_X^-$ and the same sets $A^n_i$ as before),  we have that there exists a sequence of step functions $Y_n$ of the form $Y_n=\sum_{i=1}^n y^n_i \1_{A^n_i}$, where $y_i^n=r^-_Y(b^n_i)$, such that $ ||Y_n-r_Y^-(U)|| \rightarrow 0$. 
  \\  Moreover, by Lemma \ref{lemma_30} ("only if" direction), we conclude that for each $n$, $X_n$ and $Y_n$ are comonotonic step functions. Indeed, we have $X_n = \sum_{i=1}^n x^n_i \1_{A^n_i} $ and 
    $Y_n = \sum_{i=1}^n y^n_i \1_{A^n_i} $, where  $x^n_{1} \geq x^n_{2}  \geq  \cdots \geq x^n_{n}$, and where $y^n_{1} \geq y^n_{2}  \geq  \cdots \geq y^n_{n}$, and where the ${A^n_i}$'s form a partition of $\Omega$. 
    \\ Furthermore, since $ X $ and $r_X^-(U)$ have the same distribution function (with respect to $c$), we have  by Proposition \ref{prop_38}: $ \EE_{\phi^{\G} \circ c}(X) =  \EE_{\phi^{\G} \circ c}(r_X^-(U)) $. 
   And similarly 
    $  \EE_{\phi^{\G} \circ c}(Y)=  \EE_{\phi^{\G} \circ c}(r_Y^-(U)) $. 
    Moreover, as $X$ and $ Y$ are comonotonic, we have:
 $X+Y$ and $ (r_X^-+r_Y^-)(U)$ have the same distribution function (with respect to $c$).
    Hence, by Proposition \ref{prop_38}, $  \EE_{\phi^{\G} \circ c}(X+Y) =  \EE_{\phi^{\G} \circ c}(r_X^-(U)+r_Y^-(U))$. 
    Combining this with Lemma \ref{lemma_37} and
   Remark \ref{remark_40} gives: for all $\omega \in \Omega$,
  \[
     \EE_{\phi^{\G} \circ c} (X_n)(\omega) \xrightarrow[n \rightarrow \infty]{}   \EE_{\phi^{\G} \circ c}(X)(\omega), \,\,\,\,\,\,\,\,
            \EE_{\phi^{\G} \circ c} (Y_n)(\omega) \xrightarrow[n \rightarrow \infty]{}   \EE_{\phi^{\G} \circ c}(Y) (\omega),
              \]
              and
              \[
             \EE_{\phi^{\G} \circ c} (X_n+Y_n)(\omega) \xrightarrow[n \rightarrow \infty]{}   \EE_{\phi^{\G} \circ c}(X+Y)(\omega).  \]  
             The proof is completed.\qed
             \end{proof} 

We can now proceed to proving Theorem \ref{Theorem_15}. 
The proof relies on  Lemma \ref{lemma_36} and Proposition \ref{prop_39}.
\begin{proof}[Proof of Theorem \ref{Theorem_15}]
We proceed in three steps:
    \begin{itemize}
        \item [\textbf{Step 1.}](Comonotonic additivity for non-negative step functions) We show the equality \eqref{eq156} for non-negative step functions, by using Lemma \ref{lemma_36}. Let $X$ and $Y$ be two non-negative comonotonic step functions. By Lemma \ref{lemma_30} we can represent $X$ and $Y$ as follows: 
        $X=\sum_{i=1}^n x_i\1_{A_i}$ and $Y= \sum_{i=1}^n y_i\1_{A_i}$, where  $x_{1} \geq  \cdots \geq x_{n} \geq0$ and $y_{1} \geq \cdots \geq y_{n}\geq 0$, and the $A_i$'s form a partition of $\Omega$. 
        We obtain: for all $\omega \in \Omega$,
        \[ \begin{split}
         \EE_{\phi^{\G} \circ c}(X+Y)(\omega)
         &=  \EE_{\phi^{\G} \circ c}\left(\sum_{i=1}^n x_i\1_{A_i} + \sum_{i=1}^n y_i\1_{A_i}\right)(\omega)\\
          &=  \EE_{\phi^{\G} \circ c}\left(\sum_{i=1}^n (x_i+ y_i) \1_{A_i}\right)(\omega)
            \\& = \sum_{i=1}^n  ((x_i+ y_i) - ( x_{i+1}+ y_{i+1})) \, \phi^{\G}(\omega, c(\cup_{j=1}^i A_j)),
              \end{split} \]
          where we have used Lemma \ref{lemma_36} Equation \eqref{eq153}.
        \\  We thus have:
          \[ \begin{split}         & \EE_{\phi^{\G} \circ c}(X+Y)(\omega)   = \sum_{i=1}^n  ((x_i - x_{i+1})+( y_i - y_{i+1})) \, \phi^{\G}(\omega, c(\cup_{j=1}^i A_j))
                \\& = \sum_{i=1}^n  ((x_i - x_{i+1}) \, \phi^{\G}(\omega, c(\cup_{j=1}^i A_j)) + \sum_{i=1}^n ( y_i - y_{i+1})) \, \phi^{\G}(\omega, c(\cup_{j=1}^i A_j))
                \\&= \EE_{\phi^{\G} \circ c}\left(\sum_{i=1}^n x_i\1_{A_i}\right) (\omega)+ \EE_{\phi^{\G} \circ c}\left(\sum_{i=1}^n y_i\1_{A_i}\right)(\omega)
                 \\ &= \EE_{\phi^{\G} \circ c}(X) (\omega)+ \EE_{\phi^{\G} \circ c}(Y)(\omega),
          \end{split} \]
         where we have used again Lemma \ref{lemma_36} Equation \eqref{eq153}. 
         We conclude that $ \EE_{\phi^{\G} \circ c}(X+Y) = \EE_{\phi^{\G} \circ c}(X)+ \EE_{\phi^{\G} \circ c}(Y)$.
        \item[\textbf{Step 2.}] (Comonotonic additivity for non-negative measurable functions)
        We prove the desired equality \eqref{eq156} for non-negative comonotonic measurable functions in $\chi(\F)$.  Let $X$ and $Y$ be non-negative comonotonic measurable functions in $\chi(\F)$.  
        By Proposition \ref{prop_39},
       there exist two sequences of non-negative step functions $(X_n)$ and $(Y_n)$ such that $X_n$ and $Y_n$ are comonotonic, for each $n$, and such that, for all  $\omega \in \Omega$,
         \[
     \EE_{\phi^{\G} \circ c} (X_n)(\omega) \xrightarrow[n \rightarrow \infty]{}    \EE_{\phi^{\G} \circ c}(X)(\omega) \,\,; \,\,\,\,
            \EE_{\phi^{\G} \circ c} (Y_n)(\omega) \xrightarrow[n \rightarrow \infty]{}    \EE_{\phi^{\G} \circ c}(Y) (\omega), \,\,;\,\
              \]
              \[
             \EE_{\phi^{\G} \circ c} (X_n+Y_n)(\omega) \xrightarrow[n \rightarrow \infty]{}    \EE_{\phi^{\G} \circ c}(X+Y)(\omega). 
            \]
  As $X_{n}$ and $Y_{n}$ are non-negative comonotonic step functions for each $n$, by using Step 1, we get: $  \EE_{\phi^{\G} \circ c} (X_{n}+Y_{n})(\omega) = \EE_{\phi^{\G} \circ c} (X_{n})(\omega)+ \EE_{\phi^{\G} \circ c} (Y_{n})(\omega) $, for all $\omega \in \Omega$. 
    \\ Therefore, for all $\omega \in \Omega$,
    \[ \begin{split}
    &\EE_{\phi^{\G} \circ c}(X+Y)(\omega) = \lim_{n \rightarrow \infty} \EE_{\phi^{\G} \circ c} (X_{n}+Y_{n})(\omega)
    \\&=   \lim_{n \rightarrow \infty} \EE_{\phi^{\G} \circ c} (X_{n})(\omega)+  \lim_{n \rightarrow \infty} \EE_{\phi^{\G} \circ c} (Y_{n})(\omega)
    = \EE_{\phi^{\G} \circ c}(X)(\omega) + \EE_{\phi^{\G} \circ c}(Y)(\omega).
     \end{split} \]
        \item[\textbf{Step 3.}]  
        (Comonotonic additivity for measurable functions in $\chi(\F)$)
        We prove the desired equality \eqref{eq156} for the general case.
         Let $X$ and $Y$ be  comonotonic measurable functions in $\chi(\F)$. 
         We define the measurable functions $Z_1$ and $Z_2$ by $ Z_1:= X+||X||$, $ Z_2:= Y+||Y||$, respectively.
         Then, $Z_1, Z_2 \geq 0$ and they are also comonotonic measurable functions. 
         It follows from Step 2 that for all $\omega \in \Omega$,
         $
          \EE_{\phi^{\G} \circ c}(Z_1+Z_2)(\omega) = \EE_{\phi^{\G} \circ c}(Z_1)(\omega) + \EE_{\phi^{\G} \circ c}(Z_2)(\omega)$.
         Therefore, by the translation invariance property of $ \EE_{\phi^{\G} \circ c}(\cdot)$, we have: for all $\omega \in \Omega$,
  $  
  \EE_{\phi^{\G} \circ c}(Z_1+Z_2) (\omega)=   \EE_{\phi^{\G} \circ c}(X+Y)(\omega) +||X||+||Y||, 
  $
  and $\EE_{\phi^{\G} \circ c}(Z_1)(\omega) =  \EE_{\phi^{\G} \circ c}(X+||X||) (\omega)
  = \EE_{\phi^{\G} \circ c}(X)(\omega) +||X||$, and
$\EE_{\phi^{\G} \circ c}(Z_2)(\omega) =  \EE_{\phi^{\G} \circ c}(Y+||Y||) (\omega)
  = \EE_{\phi^{\G} \circ c}(Y)(\omega) +||Y||$.
  Hence,
$   \EE_{\phi^{\G} \circ c}(X+Y)(\omega)
  =  \EE_{\phi^{\G} \circ c}(X)(\omega)+ \EE_{\phi^{\G} \circ c}(Y)(\omega)$, for all  $\omega \in \Omega$.
      \end{itemize}
   The proof is completed.\qed
\end{proof}

\section{Randomly distorted Choquet integrals and first-order stochastic dominance with respect to a capacity}\label{sec_representation_first_order}

We complete hereafter the list of properties of the randomly distorted Choquet integral.
\begin{proposition}[Properties] \label{prop_42}
Let $(\Omega, \F)$ be a measurable space.
Let $c$ be a capacity.
Let $\phi^{\G}$ be a $\G$-random distortion function.
Let $X \in \chi(\F) $.
Let us consider $\rho $ defined by:
    \[
\rho(X):=   \EE_{\phi^{\G} \circ c}(X).
\]
Then, $\rho$ is a $\G$-conditional risk measure valued in $\chi(\G)$ satisfying the following properties:
    \begin{itemize}
 \item [(i)] (Consistency in the first-order stochastic dominance with respect to the capacity $c$) If $X, Y \in  \chi(\F) $ are such that $ X \preceq_{st,c} Y $. Then,
  $\rho(X) \leq \rho(Y)$.

       \item [(ii)](Comonotonic additivity) 
Under Assumption \ref{assumption_Uniform_capacity}, 
if $c$ is continuous from below capacity, and $X$ and $Y$ are comonotonic, then $  \rho \left( X + Y \right) =  \rho \left( X \right) + \rho \left( Y \right) $.
        \item [(iii)] (Monotonicity) If $X, Y \in \chi(\F)$ are such that $ X \leq Y $. Then, $  \rho(X) \leq \rho(Y)$.
          \item [(iv)](Normalisation) $\rho(1) =1$. 
       \item [(v)] (Translation invariance)  For $a \in \R$, we have: $  \rho(a+X) = a + \rho(X)$.
       \item[(vi)](Positive homogeneity) For $a \geq 0$, we have: $  \rho(aX) = a \rho(X)$.
    \end{itemize}
\end{proposition}
\begin{proof}
    Let us prove (i).  Let $X, Y \in  \chi(\F)$ be such that $ X \preceq_{st,c} Y $. Then, for each $x\in \mathbb{R}$, we have: $c(X > x)  \leq c(Y > x)$ (by using Definition \ref{def_CDF_G} and Lemma \ref{charc_1_c}).    
      As $\phi$ is non-decreasing in the second argument, we obtain: for each $\omega \in \Omega$,   for each $x\in \mathbb{R}$,    
     $ \phi^{\G} \left( \omega, c(X > x) \right) \leq \phi^{\G} \left( \omega, c(Y > x) \right)$.  
      Hence,   
      \[ \begin{split}     \int_{0}^{+\infty} \phi^{\G} \left( \omega, c(X > x )\right) dx + \int_{-\infty}^{0} \left[ \phi^{\G} \left( \omega, c(X > x) \right) -1 \right] dx     \\ \leq         \int_{0}^{+\infty} \phi^{\G} \left( \omega, c(Y > x) \right) dx + \int_{-\infty}^{0} \left[ \phi^{\G} \left( \omega, c(Y > x) \right) -1 \right] dx. \end{split}    \]  
      Therefore,  $\rho(X)(\omega)=  \EE_{\phi^{\G} \circ c}(X)(\omega) \leq  \EE_{\phi^{\G}  \circ c}(Y)(\omega)= \rho(Y)(\omega)$.\\  
 The proof of (ii) was done in   Theorem \ref{Theorem_15}. For the proof of 
     (iii), we refer to our Proposition \ref{prop_38}.
       For (iv), by a direct computation (using that $c(\Omega)=1$), we have: $ \EE_{\phi^{\G} \circ c}(1_{\Omega})(\omega) = \phi^{\G} \left( \omega, 1 \right)= 1$, for all $\omega \in \Omega$.\\
      For the proof of (v),  we refer to  Proposition \ref{prop_38}.\\
      Let us prove (vi).  
      For $a > 0$, by Definition \ref{defn3}, we obtain: for all $\omega \in \Omega$, 
      \[\begin{split}
     &\rho(aX)(\omega) 
     = \int_{0}^{+\infty} \phi^{\G} \left( \omega, c(aX > x) \right)dx + \int_{-\infty}^{0} [\phi^{\G} \left( \omega, c(aX > x) \right) -1] dx
     \\&= 
           a \left[  \int_{0}^{+\infty} \phi^{\G} \left( \omega, c(X > z) \right)dz + \int_{-\infty}^{0} [\phi^{\G} \left( \omega, c(X > z) \right) -1] dz \right]
        = 
       a \rho(X)(\omega), \end{split}\]
        where we have used a change of variables. \qed
\end{proof}

We are now interested in conditional risk measures which can be represented as $\G$-randomly distorted Choquet integrals with respect to a $\G$-randomly distorted capacity. For the case where $\phi^{\G}$ is deterministic, the representation result \eqref{eq155} hereafter
can be found in \cite{grigorova2014risk} (Lemma 3.3).

\begin{theorem}[Representation in terms of $\G$-randomly distorted Choquet integrals]\label{Theorem_14}
 Let $(\Omega, \F)$ be a measurable space. 
 Let $c$ be a continuous from below capacity. 
 We assume Assumption \ref{assumption_Uniform_capacity}.
  Let $\rho: \chi(\F) \rightarrow  \chi(\G)$  be a conditional risk measure satisfying the following properties: 
 \begin{itemize}
   \item [(i)] Monotonicity (consistency) in the first-order stochastic dominance with respect to the capacity $c$: If $ X \preceq_{st,c} Y $, then $  \rho \left( X \right) \leq  \rho \left( Y \right) $.
    \item [(ii)] Comonotonic additivity: If $X$ and $Y$ are comonotonic, then $  \rho \left( X + Y \right) =  \rho \left( X \right) + \rho \left( Y \right) $.
  \item [(iii)] Translation invariance:  For $a \in \R$, we have: $  \rho(X+a) = \rho(X)+ a$.
       \item[(iv)] Positive homogeneity: For $a \geq 0$, we have: $  \rho(aX) = a \rho(X)$.
 \end{itemize}
Then, there exists a $\G$-random distortion function $\phi^{\G}$ such that 
\begin{equation}\label{eq155}
  \rho \left( X \right) 
   =  \EE_{\phi^{\G} \circ c}(X), \,\,\, \text{for all} \,\, X \in  \chi(\F).
\end{equation}
\end{theorem}

\begin{remark}\label{remark_25}
We recall that any $\rho: \chi(\F) \rightarrow  \chi(\G)$ satisfying the translation invariance and the positive homogeneity is normalised  ($\rho(1)=1$) and grounded ($\rho(0)=0$).  
Indeed, by the translation invariance invariance, we have: $\rho(2)=\rho(1+1)=\rho(1)+1. $ By the positive homogeneity, we have:
$\rho(2)=2\rho(1). $  Hence,  $\rho(1)=1.$  Using this, and the translation invariance again, gives: $1=\rho(1)=\rho(0+1)=\rho(0)+1.$ Hence, $\rho(0)=0.$
\end{remark}
\begin{remark} \label{remark_42}
    By the above theorem, for $X=\1_A$, where $A \in \F$, we have: for all $\omega \in \Omega$, $\rho( \1_A ) (\omega)
   =  \EE_{\phi^{\G} \circ c}(\1_A)(\omega) = \phi^{\G}(\omega,c(A))$. 
  Moreover, due to Assumption \ref{assumption_Uniform_capacity} and Lemma \ref{lemma_38}, we have: for each $t \in [0,1]$, there exists $A \in \F$ such that $c(A)=t$, and in this case $  \rho( \1_A )(\omega) =  \phi^{\G}(\omega,t)$, for all $\omega \in \Omega$. 
    In other words, under Assumption \ref{assumption_Uniform_capacity}, $\rho$ induces a random distortion function $\phi^{\G}$. 
    Furthermore, 
   $\rho$ uniquely induces $\phi^{\G}$.
    Indeed, the consistency in the first-order stochastic dominance (with respect to the capacity $c$) implies the "distribution invariance" with respect to $c$ (cf. Definition \ref{defn5}).
    Hence, for each $t \in [0,1]$, if $A,B \in \F$ are such that $c(A)=c(B)=t$, then by the "distribution invariance" of $\rho$, we get $\rho(\1_A)(\omega)=\rho(\1_B)(\omega)= \phi^{\G}(\omega,t)$, for all $\omega \in \Omega$.
\end{remark}

\begin{proof}[Proof of Theorem \ref{Theorem_14}]
     Let $t \in [0,1]$. 
     There exists $A \in \F$ such that $c(A)=t$. 
         Indeed, if we set $A:=\{U>1-t\}$, where $U:=G_Z(Z)$, then, according to Lemma \ref{lemma_38}, we have: $c(A)=c(U>1-t)=1-G_U(1-t)= 1- (1-t)=t$. 
    We define $\phi^{\G}$ as follows (cf. Remark \ref{remark_42}):  for all $ \omega \in \Omega$,
    \begin{equation}\label{eq152} 
         \phi^{\G}(\omega,t)=\phi^{\G}(\omega,c(A))= \rho(\1_A)(\omega).
    \end{equation}
    It is easy to check that $\phi^{\G}$ is a $\G$-random distortion function. 
    Indeed, for each $t \in [0,1]$,
    $\phi^{\G}(\cdot,t)$ is $\G$-measurable (as $\rho(\1_A)$ is $\G$-measurable). 
    Moreover, for $t=0$, for all $\omega \in \Omega$, $\phi^{\G}(\omega,0)=\phi^{\G}(\omega,c(\emptyset))= \rho(\1_{\emptyset})(\omega)=\rho(0)(\omega)=0$ (as $c(\emptyset)=0$ and as $\rho$ is grounded). 
    Similarly, for $t=1$, for all $\omega \in \Omega$, $\phi^{\G}(\omega,1)=\phi^{\G}(\omega,c(\Omega))= \rho(\1_{\Omega})(\omega)=\rho(1)(\omega)=1$ (as $c(\Omega)=1$ and as $\rho$ is normalised). 
    Let $0 \leq s \leq t \leq 1$. 
    There exist $A \in \F$ and  $B \in \F$  such that $B \subseteq A$, $c(B)=s$ and $c(A)=t$. 
    Indeed, by setting $B:=\{U>1-s\}$ and $A:=\{U>1-t\}$, 
    where $U:=G_Z(Z)$ (as Lemma \ref{lemma_38}), we have that $B \subseteq A$. 
    Moreover, applying Lemma \ref{lemma_38}, we have: $c(B)=s$ and $c(A)=t$. 
    Hence, for all $\omega \in \Omega$,
    \[
    \phi^{\G}(\omega,s)=\phi^{\G}(\omega,c(B))= \rho(\1_B)(\omega) 
    \leq  \rho(\1_A)(\omega)=\phi^{\G}(\omega,c(A))= \phi^{\G}(\omega,t), 
    \]
    where we have used the monotonicity of $\rho$. Hence, $\phi^{\G}$ is a $\G$-random distortion function. 
     \\
     It follows from Property (i) that $\rho$ is "distribution invariant". Hence, $\phi^{\G}$ is well-defined and in fact is unique (cf. Remark \ref{remark_42}). 
    The sequel of the proof is divided in four steps:
    \begin{itemize}
    \item [\textbf{Step 1.}] For $X = \1_{A}$, where $A \in \F$, we have, by using \eqref{eq152}, 
     \[
   \rho(\1_A)(\omega) =\phi^{\G}(\omega,c(A))= \EE_{\phi^{\G}\circ c}(\1_A)(\omega),
   \,\, \text{for all} \,\, \omega \in \Omega.
    \]
    \item [\textbf{Step 2.}] Let $X$ be a non-negative step function of the form  $X = \sum_{i=1}^n x_i\1_{A_i}$,  
    where $x_1 \geq x_2 \geq \cdots \geq x_n \geq 0$ and $ A_i \in \F$ such that the $A_i$'s form a partition of $\Omega$. 
    Then, $X$ can be rewritten as $X = \sum_{i=1}^n \Bar{x}_i \1_{\Bar{A}_i}$,
       where $ \Bar{x}_i:= x_i - x_{i+1},  \, x_{n+1} := 0$, and $ \Bar{A}_i := \cup_{k=1}^i A_k$. 
    By the comonotonic additivity and positive homogeneity of $\rho$, we have: for all $\omega \in \Omega$,
    \[
    \rho(X)(\omega)
    =\rho \left(\sum_{i=1}^n \Bar{x}_i\1_{\Bar{A}_i}\right)(\omega)
    =\sum_{i=1}^n \Bar{x}_i \rho(\1_{\Bar{A}_i})(\omega).
    \]
    On the other hand, by Lemma \ref{lemma_36}, 
\[
 \EE_{\phi^{\G} \circ c}(X)(\omega)
 = \sum_{i=1}^n  \Bar{x}_i \, \phi^{\G}(\omega, c(\Bar{A}_i)), \,\, \text{for all} \,\, \omega \in \Omega.
\]
By Step 1, for each $i$, $\phi^{\G}(\omega,c(A_i))=\rho(\1_{A_i})(\omega)$, for all $ \omega \in \Omega$. Hence, 
\[ 
\rho(X)(\omega) = \rho \left(\sum_{i=1}^n \Bar{x}_i\1_{\Bar{A}_i}\right)(\omega)
= \sum_{i=1}^n  \Bar{x}_i \, \phi^{\G}(\omega, c(\Bar{A}_i))
= \EE_{\phi^{\G} \circ c}(X) (\omega), \,\, \text{for all} \, \omega \in \Omega.
\]

  \item [\textbf{Step 3.}] 
  Let $X$ be a non-negative measurable function in $\chi(\F)$. 
  As $c$ is continuous from below, we apply Proposition \ref{prop_39}.
By Proposition \ref{prop_39}, there exists a sequence of non-negative step functions $(X_n)$ such that $X_n \rightarrow r_X^-(U)$ in $||\cdot||$. Moreover, $X$ and $r_X^-(U)$ have the same distribution. 
  We apply Lemma \ref{lemma_37} and Remark \ref{remark_32}, and the "distribution invariance" of $\rho(\cdot)$ and $\EE_{\phi^{\G} \circ c}(\cdot)$. 
  \\ By Lemma \ref{lemma_37} (the Lipschitz continuity of $\EE_{\phi^{\G} \circ c}(\cdot)$) and the "distribution invariance" of $\EE_{\phi^{\G} \circ c}(\cdot)$, we have: $ \EE_{\phi^{\G} \circ c} (X_n) \xrightarrow[n \rightarrow \infty]{}   \EE_{\phi^{\G} \circ c}(X)$ in $||\cdot||$. 
  Hence, 
\begin{equation}\label{eq157}
  \EE_{\phi^{\G} \circ c} (X_n)(\omega) \xrightarrow[n \rightarrow \infty]{}    \EE_{\phi^{\G} \circ c}(X)(\omega),\,\, \text{for all} \,\, \omega \in \Omega.
\end{equation}
 By Remark \ref{remark_32}, we have that $\rho$ is Lipschitz continuous. 
 Hence, using the distribution-invariance of $\rho$ (and the fact that $ r_X^-(U)$ and $X$ have the same distribution), we have
  $ \rho(X_n) \xrightarrow[n \rightarrow \infty]{}   \rho (X)$ in $||\cdot||$.
  Therefore, 
\begin{equation}\label{eq158}
  \rho (X_n)(\omega) \xrightarrow[n \rightarrow \infty]{}    \rho(X)(\omega), \,\, \text{for all} \,\, \omega \in \Omega.
  \end{equation}
  On the other hand, by Step 2, we have: $\rho(X_n)(\omega)=\EE_{\phi^{\G} \circ c} (X_n)(\omega)$, for all $\omega \in \Omega$, for each $n$.
  Using this, and equations \eqref{eq157} and \eqref{eq158}, we get: 
  $\rho (X)(\omega)= \EE_{\phi^{\G} \circ c}(X)(\omega)$, for all $\omega \in \Omega$.
  \item [\textbf{Step 4.}]
   Let $X \in \chi(\F)$. We define $ Y:= X+||X||$, then $Y \geq 0$. It follows from Step 3 that 
$    \rho (Y)(\omega)= \EE_{\phi^{\G} \circ c}(Y)(\omega)$, for all $\omega \in \Omega$.
By the translation invariance property of $\rho(\cdot)$ and $\EE_{\phi^{\G} \circ c}(\cdot)$, we get the desired result 
 $  \rho (X)(\omega) = \EE_{\phi^{\G} \circ c}(X)(\omega) $, for all $\omega \in \Omega$. \qed
 \end{itemize}
 \end{proof}

\section{Concave $\G$-random distortion functions and the stop-loss stochastic dominance with respect to a capacity}\label{sec_representation_stop_loss}

In this section, we first recall the definition of the increasing convex stochastic ordering (with respect to a capacity $c$) and 
the definition of the stop-loss ordering (with respect to a capacity $c$) (cf. \cite{grigorova2014stochastic}).

 \begin{definition} \label{def_2stoch_domin_c} 
Let $X$ and $Y$ be in $\chi(\F)$ and let $c$ be a capacity. We say that $X$ is dominated by  $Y$ in the increasing convex stochastic dominance with respect to the capacity $c$, denoted by $ X \preceq_{icx,c} Y $ if :
\begin{equation}
      \mathbb{E}_c \left( u\left( X \right ) \right ) \leq \mathbb{E}_c \left( u\left( Y \right )  \right ) 
\end{equation}
for all deterministic functions $u: \mathbb{R} \rightarrow \mathbb{R} $ which are non-decreasing and convex.
\end{definition}

 \begin{definition} \label{def_stop_loss_domin_c} 
Let $X$ and $Y$ be in $\chi(\F)$ and let $c$ be a capacity.
We say that $X$ is dominated by  $Y$ in the stop-loss ordering with respect to the capacity $c$, denoted by $ X \preceq_{sl,c} Y $ if :
\begin{equation}
     \mathbb{E}_c \left(\left(X-b \right)^+ \right) \leq  \mathbb{E}_c \left(\left(Y-b \right)^+ \right), 
      \,\, \text{for all} \,\, b \in \mathbb{R}.
\end{equation}
\end{definition}

The increasing convex stochastic dominance (with respect to a capacity) implies the stop-loss stochastic dominance (with respect to a capacity).
The following characterisation of the stop-loss ordering relation with respect to a capacity can be found in \cite{grigorova2014stochastic}. 
  \begin{proposition}  \label{Charc_2orderings_c}
     Let $c$ be a capacity. 
      Let $X$ and $Y$ be in $\chi(\F)$. 
      The following statements are equivalent: 
 \begin{enumerate}
  \item[(i)]  
   $  \mathbb{E}_c \left( \left( X-b \right )^+ \right ) \leq  \mathbb{E}_c \left( \left( Y-b \right )^+  \right )$
     for all $b \in \mathbb{R}$.
  
   \item[(ii)]  $ \int_{x}^{+\infty} \left( 1-G_X\left( z \right ) \right ) dz  \leq \int_{x}^{+\infty} \left( 1-G_Y\left( z \right ) \right ) dz $
 for all $x \in \mathbb{R}$.
 
\item[(iii)]  $ \int_{\alpha}^{1} r_X\left( u \right ) du  \leq \int_{\alpha}^{1}r_Y\left( u \right )  du $ 
for all $\alpha \in \left( 0 , 1 \right) $. 

    \item[(iv)] 
 For all $g: [0,1] \rightarrow \R^+$ integrable and non-decreasing functions, we have:
   \[ \int_{0}^{1} g(t) r_X(t)dt  \leq \int_{0}^{1} g(t)r_Y(t)dt. \]
     \end{enumerate}
  \end{proposition}

  The following proposition is a generalisation of property (iv), when the function $g$ is a random function (depending on $\omega$ and $t$). It will be used in the sequel.
\begin{proposition} \label{prop_37_c}
 Let $c$ be a capacity. 
Let $ X, Y \in \chi(\F)$.
    The following statements are equivalent: 
 \begin{enumerate}
 \item[(i)] For all $g: [0,1] \rightarrow \R^+$ integrable and non-decreasing, we have:
   \[ \int_{0}^{1} g(t) r_X(t)dt  \leq \int_{0}^{1} g(t)r_Y(t)dt. \]
 \item[(ii)] 
 For all random functions $g: \Omega \times [0,1] \rightarrow \R^+$, which are integrable and non-decreasing (in the second argument), we have: for all $\omega \in \Omega$,
   \[ \int_{0}^{1} g(\omega, t) r_X(t)dt  \leq \int_{0}^{1} g(\omega, t)r_Y(t)dt. \]
     \end{enumerate}
\end{proposition}
\begin{proof}  
Let $\omega \in \Omega$ be fixed. The implication $(ii) \Rightarrow (i)$ is trivially obtained by taking $g(\omega,t) := g(t)$, which is non-negative, non-decreasing and integrable. 
\\ Let us now prove the converse implication. 
Let us fix $\omega$.
Since $\omega$ is fixed, the function $t \mapsto g(\omega,t)$ is non-negative, integrable, 
and non-decreasing.
Hence, by applying (i) with the function $t \mapsto g(\omega,t)$ (for a fixed $\omega \in \Omega$) we get:   \[ \int_{0}^{1} g(\omega, t) r_X(t)dt  \leq \int_{0}^{1} g(\omega, t)r_Y(t)dt. \]
As this can be done for all  $\omega \in \Omega$, the result holds. \qed
\end{proof}

\begin{definition}
    We say that a  $\G$-random distortion function $\phi^{\G}$ is \textbf{concave} if: for all $\omega \in \Omega$, the function $t  \mapsto \phi(\omega,t)$ is \textbf{concave}.
\end{definition}
Let $\phi^{\G}$ be a concave $\G$-random distortion function and let $c$ be a capacity. 
Let $X \in \chi(\F) $.
Let us consider $\rho $ defined by:
$\rho(X):=   \EE_{\phi^{\G} \circ c}(X)$.
Then, $\rho$ is a $\G$-conditional risk measure valued in $ \chi (\G)$ which satisfies all the properties of Proposition \ref{prop_42} 
(as $\phi^{\G}$ is a random distortion function).
Moreover, the following proposition holds.
\begin{proposition} \label{prop_26_3_c}
 Let $c$ be a capacity. 
   Let $\phi^{\G}$ be a \textbf{concave} $\G$-random distortion function.
Let us consider $\rho $ defined by:
$\rho(X):=   \EE_{\phi^{\G} \circ c}(X)$, for all $X \in \chi(\F) $.
Then, $\rho$ is monotone (or consistent) with respect to the $\preceq_{sl,c}$-stochastic dominance,   
that is, if $X, Y \in  \chi(\F) $ are such that $ X \preceq_{sl,c} Y $, then $\rho(X)(\omega) \leq \rho(Y)(\omega),$ for all $\omega\in\Omega$. 
\end{proposition}

In order to prove Proposition \ref{prop_26_3_c}, we will introduce the following lemma. 
This lemma is well-known in the case of a capacity $c$ and a distortion function $\phi^{\G}$ which is deterministic (cf. \cite{grigorova2014risk}). 
In the case where $c=P$ is a probability measure and $\phi^{\G}$ is deterministic, this lemma is by now a standard result in (convex) risk measures.

\begin{lemma}\label{lemma_35_c}
    Let $c$ be a capacity. 
   Let $\phi^{\G}$ be a\textbf{ concave} $\G$-random distortion function. 
   For all $X \in \chi(\F)$, for all $\omega \in \Omega$, 
      \begin{equation}\label{eq146_c}
         \EE_{\phi^{\G} \circ c}(X)(\omega) = \phi^{\G}(\omega, 0+) \, \sup_{t<1} r^+_X(t) + \int_{0}^1 {\phi}^{'}(\omega, 1-t) r_X^{+}(t) dt
      \end{equation}
\end{lemma}
\begin{proof}
The proof of the lemma follows the same arguments as Lemma 3.16 in \cite{grigorova2014risk} and is provided for the convenience of the reader.
    It suffices to prove \eqref{eq146_c} for non-negative measurable functions in $\chi(\F)$ as the terms on both sides of \eqref{eq146_c} are translation invariant. 
    Let $X \geq 0$. 
   Due to the nondecreasingness of $G_X$ and to the definition of $r_X^+$, we have: 
    \begin{equation}\label{eq147_c}
         r_X^{+}(t)= \int_{0}^{\infty} \1_{\{G_X(s)\leq t\}} ds. \end{equation}
    For $\omega$ fixed, let $\phi^{'}(\omega,\cdot)$ denote the right-hand derivative of the concave function $t \mapsto \phi(\omega,t)$ on $(0,1)$.
   Let $\omega \in \Omega$ be fixed, we obtain:
    \[\begin{split}
        &\int_{0}^1 {\phi}^{'}(\omega, 1-t) r_X^{+}(t) dt
        = \int_{0}^1 {\phi}^{'}(\omega, 1-t) \int_{0}^{\infty} \1_{\{G_X(s)\leq t\}} ds  dt\\
      &  = \int_{0}^{\infty} \int_{0}^1 {\phi}^{'}(\omega, 1-t)  \1_{\{G_X(s)\leq t\}} dt ds 
           = \int_{0}^{\infty} \int_{0}^1 {\phi}^{'}(\omega, z)  \1_{\{z \leq 1- G_X(s)\}} dz ds\\ 
           & = \int_{0}^{\infty} \int_{0}^{1- G_X(s)} {\phi}^{'}(\omega, z) dz ds \\
             &= \int_{0}^{\infty} \left( {\phi^{\G}}(\omega, 1- G_X(s)) - {\phi^{\G}}(\omega, 0^+) \right) \1_{\{ G_X(s) < 1\}}  ds. 
    \end{split}\]
    Hence, by the definition of the $\G$-randomly distorted Choquet integral and the fact that 
    \[
    \sup_{t<1} r_X^{+}(t)= \int_{0}^{\infty} \1_{\{G_X(s) <1 \}} ds, 
    \]
    we obtain: 
       \[\begin{split}
        &\int_{0}^1 {\phi}^{'}(\omega, 1-t) r_X^{+}(t) dt
        = \int_{0}^{\infty}  {\phi^{\G}}(\omega, 1- G_X(s))  ds - {\phi^{\G}}(\omega, 0^+) \int_{0}^{\infty} \1_{\{ G_X(s) < 1\}}  ds 
        \\ &=  \EE_{\phi^{\G} \circ c}(X)(\omega) - {\phi^{\G}}(\omega, 0^+) \sup_{t<1} r_X^+(t)
         =  \EE_{\phi^{\G} \circ c}(X)(\omega) - {\phi^{\G}}(\omega, 0^+) r_X^+(1).  \end{split} 
         \]
         \end{proof}\qed

\begin{proof}[Proof of Proposition \ref{prop_26_3_c}]
    Let $X, Y \in  \chi(\F)$ be such that $ X \preceq_{sl,c} Y $. 
    Let us show that $ \EE_{\phi^{\G} \circ c}(X) \leq  \EE_{\phi^{\G} \circ c}(Y)$. 
    Let $\omega \in \Omega$ be fixed.
    Due to Lemma \ref{lemma_35_c}, it is enough to show that 
    \[
    \phi^{\G}(\omega, 0+) \, \sup_{t<1} r_X^+(t) + \int_{0}^1 {\phi}^{'}(\omega, 1-t) r_X^{+}(t) dt
    \leq
    \phi^{\G}(\omega, 0+) \, \sup_{t<1} r_Y^+(t) \int_{0}^1 {\phi}^{'}(\omega, 1-t) r_Y^{+}(t) dt.
    \]
  By using Proposition \ref{prop_37_c}, we have: 
      \[
   \int_{0}^1 {\phi}^{'}(\omega, 1-t) r_X^{+}(t) dt
    \leq
   \int_{0}^1 {\phi}^{'}(\omega, 1-t) r_Y^{+}(t) dt.
    \]
    As $ \phi^{\G}(\omega, 0+) \geq 0$, it remains to show that $ \sup_{t<1} r_X^+(t) \leq \sup_{t<1} r_Y^+(t)$. 
   By contradiction, assume that $ \sup_{t<1} r_X^+(t) > \sup_{t<1} r_Y^+(t)$. Then, there exists $t_0 \in [0,1)$ such that 
    \[
    r_X^+(s) \geq r_X^+(t_0) > \sup_{t<1} r_Y^+(t)  \,\,\,\, \text{for all} \,\, s \geq t_0.
    \]
    This implies that $ r_X^+(s) > r_Y^+(s)$, for all $s \geq t_0$.
    Hence, we get:    
    $\int_{t_0}^{1} r_X^+(s) - r_Y^+(s) \, ds >0$,
    which contradicts the fact that $X \preceq_{sl,c} Y$ (cf. Theorem \ref{Charc_2orderings_c} (iii)).
    Hence, $X \preceq_{sl,c} Y$  implies that $ \sup_{t<1} r_X^+(t) \leq \sup_{t<1} r_Y^+(t)$.\qed
\end{proof}


\begin{theorem}[Representation with respect to a concave random distortion] 
Let $(\Omega, \F)$ be a measurable space. 
 Let $c$ be a continuous from below. 
 We assume Assumption \ref{assumption_Uniform_capacity}.
  Let $\rho:\chi(\F)  \rightarrow \chi(\G) $ be a conditional risk measure satisfying the following properties: 
 \begin{itemize}
   \item [(i)] Monotonicity (consistency) in the stop-loss stochastic dominance with respect to the capacity $c$: If $ X \preceq_{sl,c} Y $, then $  \rho \left( X \right) \leq  \rho \left( Y \right) $. 
    \item [(ii)] Comonotonic additivity: If $X$ and $Y$ are comonotonic, then $  \rho \left( X + Y \right) =  \rho \left( X \right) + \rho \left( Y \right) $.
  \item [(iii)] Translation invariance:  For $a \in \R$, we have: $  \rho(X+a) = \rho(X)+ a$.
       \item[(iv)] Positive homogeneity: For $a \geq 0$, we have: $  \rho(aX) = a \rho(X)$.
 \end{itemize}
Then, there exists a \textbf{concave} $\G$-random distortion function $\phi^{\G}$ such that 
\[
  \rho \left( X \right) 
   =  \EE_{\phi^{\G} \circ c}(X), \,\,\, \text{for all} \,\, X \in  \chi(\F).
\] 
\end{theorem} 

\begin{proof}
   Since the consistency in the $\preceq_{sl,c}$ stochastic dominance 
   implies the consistency in the $\preceq_{st,c}$ stochastic dominance (cf. \cite{grigorova2014risk}), 
   we apply Theorem \ref{Theorem_14}. 
    By Theorem \ref{Theorem_14}, there exists a $\G$-random distortion function $\phi^{\G}$ such that 
$  \rho \left( X \right) 
   =  \EE_{\phi^{\G} \circ c}(X)
   \,\,\, \text{for all} \,\, X \in  \chi(\F).
$ 
It remains to show that the $\G$-random distortion function $\phi^{\G}$ is concave. 
By Assumption \ref{assumption_Uniform_capacity}, we have the following: for any $s,t \in [0,1]$, there exist $A,B,C \in \F$ such that $c(A)=s$, $c(B)=t$, $A \subset B$, and $c(C)=\frac{s+t}{2}$.
We set
 $X:=\frac{1}{2}\1_{A} + \frac{1}{2}\1_{B}$ and $Y:=\1_C$.
 By the proof of Theorem 3.7. in \cite{grigorova2014risk}, we have $X \preceq_{sl,c} Y$.
The consistency with respect to $\preceq_{sl,c}$ of $\rho$ implies that $\rho(X) \leq \rho(Y)$. Hence, for all $\omega \in \Omega$, we have:
$
 \EE_{\phi^{\G} \circ c}\left(\frac{1}{2}\1_{A} + \frac{1}{2}\1_{B}\right)(\omega)
 \leq
  \EE_{\phi^{\G} \circ c}\left(\1_C\right)(\omega).
$
As $\frac{1}{2}\1_{A}$ and $\frac{1}{2}\1_{B}$ are comonotonic (as $A \subset B$),
by the positive homogeneity and comonotonic additivity of $\EE_{\phi^{\G} \circ c}(\cdot)$, we get: 
\[
 \EE_{\phi^{\G} \circ c}\left(\frac{1}{2}\1_{A} + \frac{1}{2}\1_{B}\right)(\omega)
 =
\frac{1}{2} \EE_{\phi^{\G} \circ c}(\1_{A})(\omega) + \frac{1}{2} \EE_{\phi^{\G} \circ c}(\1_{B})(\omega)
 \leq
  \EE_{\phi^{\G} \circ c}(\1_C)(\omega).
\]
Hence, 
$
\frac{1}{2} \phi^{\G} (\omega, c(A)) + \frac{1}{2} \phi^{\G} (\omega, c(B))
 \leq
 \phi^{\G} (\omega, c(C)).
$
The concavity of $\phi^{\G}$ follows
as $c(A)=s$, $c(B)=t$, $c(C)=\frac{s+t}{2}$ and as $s$ and $t$ are arbitrary. \qed
\end{proof}

\section{Examples} \label{sec_Examples}

In this section, we will present some extensions of well-known distortion risk measures to our framework.
From a mathematical point of view, we will "randomise" the distortion function.

\subsection{A "randomised" Value at Risk}

In a probabilistic framework,
the static Value at Risk (VaR) at level $\alpha \in (0,1)$ of a given \textit{potential loss} $X \in \chi(\F)$ is defined by: 
$VaR_{\alpha}(X)=q^-_X(\alpha)$.  
The deterministic distortion function induced by $VaR_{\alpha}$ is of the following form $\psi(x)=\1_{(1-\alpha,1]}(x)$ for all $x \in [0,1]$.  
The same sign convention in the definition of the VaR as the one used in this paper is used, for instance, in 
\cite{embrechts2014academic}, 
\cite{denuit2006actuarial}, 
\cite{song2009risk} 
or \cite{dhaene2006risk}. 
The probabilistic VaR has been extended to a capacity framework in \cite{grigorova:tel-00878599}, under the name GVaR (for Generalized VaR). 
\paragraph{Example 1}
Let us consider the simple case where the $\sigma$-algebra $\G$ is $\G=\{\emptyset, \Omega, A, A^c\}$. Our framework allows to have two different levels for the VaR, depending on whether $\omega$ is in $A$ or in $A^c$, say $\alpha$ and $\beta$, where $\alpha < \beta$. 
We define a random distortion function $\phi^{\G}$ as follows: for all $t \in [0,1]$,
\[
\phi^{\G}(\omega,t):= \left\{ \begin{array}{cc}
    \1_{(1-\alpha,1]}(t),   &  \text{if}\,\, \omega \in A\\
   \1_{(1-\beta,1]}(t),   &  \text{if} \,\, \omega \in A^c.  
 \end{array} \right.
\]
Let $c$ be a given capacity (modelling ambiguity).
For the above choice of $\phi^{\G}$, we consider the randomly distorted Choquet integral 
 $\EE_{\phi^{\G} \circ c}(X)$, for  $X \in \chi$.
 We will compute $\EE_{\phi^{\G} \circ c}(X)$ in the simple case where $X$ is defined as $X:=\1_C$, with $C \in \F$.
\begin{itemize}
    \item[\textbf{a.}] The probabilistic framework \\
    In the particular case where $c$ is a probability measure, we set: $P(C)=p$, where $p \in [0,1]$.
We obtain: 
\begin{equation*}
\begin{aligned}
\EE_{\phi^{\G} \circ P} (X)(\omega)&= \EE_{\phi^{\G} \circ P}(\1_C)(\omega) = \phi^{\G}(\omega,P(C))  = \phi^{\G}(\omega,p)\\
&=\1_{(1-\alpha,1]}(p)\1_A(\omega)+
   \1_{(1-\beta,1]}(p)\1_{A^c}(\omega).  
 \end{aligned}
\end{equation*}
Hence, as $\alpha < \beta$, we have: 
\begin{enumerate}
    \item  If $p \leq 1- \beta$, then 
$ \EE_{\phi^{\G} \circ P} (X)(\omega)=0$, for all $\omega \in \Omega$.

    \item If $1-\beta < p \leq 1-\alpha$, then 
  $  \EE_{\phi^{\G} \circ P} (X)(\omega)=\1_{A^c}(\omega)$, for all $\omega \in \Omega$.

      \item If $p > 1-\alpha$, then $ \EE_{\phi^{\G} \circ P} (X)(\omega)=1$, for all $\omega \in \Omega$.
\end{enumerate}
 \textbf{An application:}
  Let us consider two experts who agree on the class of risk measures to be used, in this case $\{VaR_{\alpha}: \alpha \in (0,1)\}$ but who disagree on the acceptable level of risk within this class of risk measures, either $\alpha$ or $\beta$ (with $\alpha < \beta$).
  The two experts consider that they face a random loss of the form $X=\1_C$, and agree on the probability $P(C)=p$.
 By the definition of the VaR, 
if $p \leq 1- \beta$, then both experts attribute a risk of $0$ to $X$.
Also, if $p > 1-\alpha$, then both experts attribute a risk of $1$ to $X$.
But if $1-\beta < p \leq 1-\alpha$, then they disagree: the first expert attributes $0$ and the second expert attributes $1$.
\\
Using the randomly distorted Choquet integrals, with $\phi^{\G}$ as above, allows us to randomise the risk measure in the situation where the two experts disagree (that is, where $1-\beta < p \leq 1-\alpha$): we use the binary random variable $\1_{A^c}$ instead of a fixed number (when $1-\beta < p \leq 1-\alpha$).
In this interpretation, depending on how confident a decision maker is in the first expert or in the second expert, the decision maker can take a different set $A$.
In some sense, $A$ is the set where the decision maker considers that the first expert is "right" and $A^c$ is the set where the decision maker considers that the second expert is "right" (in the situation when the two experts disagree on the level of risk to be used). Thus,  $\G=\{\emptyset, \Omega, A, A^c\}$ can be interpreted as the information which the decision maker has on the experts.  

\begin{remark}
    Let us consider the particular case where $C=A^c$, that is, $X=\1_{A^c}$. Then, we have: \begin{enumerate}
    \item  If $p=P(A^c) \leq 1- \beta$, then 
$ \EE_{\phi^{\G} \circ P} (X)(\omega)=0$, for all $\omega \in \Omega$.
    \item If $1-\beta < p=P(A^c)  \leq 1-\alpha$, then 
  $  \EE_{\phi^{\G} \circ P} (X)(\omega)=\1_{A^c}(\omega)=X(\omega)$, for all $\omega \in \Omega$.
      \item If $p=P(A^c)  > 1-\alpha$, then $ \EE_{\phi^{\G} \circ P} (X)(\omega)=1$, for all $\omega \in \Omega$.
\end{enumerate}
\end{remark}
\item[\textbf{b.}] The capacity framework \\
Let $c$ be a capacity. We set $c(C)=\gamma$, where $\gamma \in [0,1]$ (as $c$ is normalised). 
We obtain: 
\begin{align*}
\EE_{\phi^{\G} \circ c} (X)(\omega)= \EE_{\phi^{\G} \circ c}(\1_C)(\omega) = \phi^{\G}(\omega,c(C))  = \phi^{\G}(\omega,\gamma)\\
=    \1_{(1-\alpha,1]}(\gamma)  \1_A(\omega)+
   \1_{(1-\beta,1]}(\gamma)\1_{A^c}(\omega).  
\end{align*}
Hence, as $\alpha < \beta$, we have: 
\begin{enumerate}
    \item  If $\gamma \leq 1- \beta$, then 
$ \EE_{\phi^{\G} \circ c} (X)(\omega)=0$, for all $\omega \in \Omega$.

    \item If $1-\beta < \gamma \leq 1-\alpha$, then 
  $  \EE_{\phi^{\G} \circ c} (X)(\omega)=\1_{A^c}(\omega)$, for all $\omega \in \Omega$.

      \item If $\gamma > 1-\alpha$, then $ \EE_{\phi^{\G} \circ c} (X)(\omega)=1$, for all $\omega \in \Omega$.
\end{enumerate}

\textbf{An application:}
This modelling framework is suitable for the case where there is an ambiguity on the distribution of the potential losses. 
Hence, random events are assessed by the capacity $c$ instead of a given probability measure $P$. 
  Let us consider two experts who agree on the class of risk measures, say $\{GVaR_{\alpha}: \alpha \in (0,1)\}$.
  These risk measures were introduced in \cite{grigorova:tel-00878599} and defined by $GVaR_{\alpha}(X):=r_X^-(\alpha)$, for $\alpha \in (0,1)$.
  However, the two experts disagree on the acceptable level (or acceptable parameter), either $\alpha$ or $\beta$ (with $\alpha < \beta$). 
  They face a potential loss of the simple form $X=\1_C$, with $c(C)=\gamma$.
 By the definition of the GVaR, we have $GVaR_{\alpha}(\1_C)   = \1_{(1- \alpha,1]}(c(C)) = \1_{(1- \alpha,1]}(\gamma) $ and  $GVaR_{\beta}(\1_C)= \1_{(1- \beta,1]}(c(C)) = \1_{(1- \beta,1]}(\gamma) $.
 Thus, if $\gamma \leq 1- \beta$, both analysts attribute  $0$ to $X$.
If $\gamma > 1-\alpha$, then both analysts attribute  $1$ to $X$.
But if $1-\beta < \gamma \leq 1-\alpha$, then they disagree: the first analyst attributes $0$ and the second analyst attributes $1$.
\\
Using the randomly distorted Choquet integral, with $\phi^{\G}$ as above, allows the decision maker to randomise the risk measure in the situation where the two analysts disagree (where $1-\beta < \gamma \leq 1-\alpha$): we use the indicator function $\1_{A^c}$ instead of a fixed number (when $1-\beta < \gamma \leq 1-\alpha$).
Depending on how "confident" the manager is in the opinion of the first analyst or of the second analyst, they can take a different $A$.

\item [\textbf{c.}] The case where $\G=\sigma(\{A_1, A_2, \cdots, A_n\})$ 
\\
Let us consider the case where the $\sigma$-algebra $\G$ is
$\G=\sigma(\{A_1, A_2, \cdots, A_n\})$, where the $A_i$'s form a partition of $\Omega$. 
Our framework allows  different levels, 
say $\alpha_1, \alpha_2, \cdots, \alpha_n$,  where $\alpha_1 < \alpha_2 < \cdots < \alpha_n$. 
We define a random distortion function $\phi^{\G}$ as follows: for all $t \in [0,1]$, for all  $i \in \{1, \cdots, n\}$, 
$\phi^{\G}(\omega,t)= \1_{(1-\alpha_{i},1]}(t)$, if $ \omega \in A_i$; that is, $\phi^{\G}(\omega,t)=  \sum_{i=1}^n  \1_{(1-\alpha_{i},1]}(t) \1_{A_i}(\omega)$. 

Let $c$ be a capacity.
For the above choice of $\phi^{\G}$, we consider the randomly distorted Choquet integral 
 $\EE_{\phi^{\G} \circ c}(X)$, for  $X \in \chi$. By our results, we know that this operator is comonotonic additive and monotone with respect to the first-order stochastic dominance (with respect to the capacity $c$).  
Let us compute $\EE_{\phi^{\G} \circ c}(X)$ when $X$ is defined by $X:=\1_C$, with $C \in \F$.
Let $c(C)=\gamma$, where $\gamma \in [0,1]$. 
We obtain: 
\begin{align*}
\EE_{\phi^{\G} \circ c} (X)(\omega)= \EE_{\phi^{\G} \circ c}(\1_C)(\omega) = \phi^{\G}(\omega,c(C))  = \phi^{\G}(\omega,\gamma)\\
= \sum_{i=1}^n  \1_{(1-\alpha_{i},1]}(\gamma) \1_{A_i}(\omega).
\end{align*}
As $\alpha_1 < \alpha_2 < \cdots < \alpha_n$, we have: 
\begin{itemize}
    \item  If $\gamma \leq 1- \alpha_{n}$, then 
$ \EE_{\phi^{\G} \circ c} (X)(\omega)=0$, for all $\omega \in \Omega$.

   \item For $k \in \{2,\cdots, n\}$, if $1-\alpha_{k} < \gamma \leq 1-\alpha_{k-1}$, then 
  $  \EE_{\phi^{\G} \circ c} (X)(\omega)= \sum_{i=k}^{n} \1_{A_i}(\omega)$  for all $\omega \in \Omega $.
  
      \item If $\gamma > 1-\alpha_{1}$, then $ \EE_{\phi^{\G} \circ c} (X)(\omega)=1$, for all $\omega \in \Omega$.
\end{itemize}
We thus get the following expression: for all $\omega \in \Omega$,
 \[  \EE_{\phi^{\G} \circ c} (X)(\omega)= \sum_{k=1}^{n} \1_{\{1-\alpha_{k}< \gamma \leq 1-\alpha_{k-1}\}}  \1_{\cup_{i=k}^{n} A_i}(\omega),  \,\,\,\, \alpha_{0}:= 0.\]
 \textbf{An application:}
 We are in the context of model ambiguity, captured by a capacity $c$. 
  We consider a team of risk analysts, here $n$ analysts, who agree on the capacity $c$, and on the class of risk measures to be used, say $\{GVaR_{\alpha}: \alpha \in (0,1)\}$. 
  But they all disagree on the parameter $\alpha$ with $\alpha_1 < \alpha_2 < \cdots < \alpha_n$, where $\alpha_i$ is the suggested parameter by analyst $i$.
 \\ 
 Using the randomly distorted Choquet integral, with $\phi^{\G}$ as above, allows us to "randomise" the risk measure on the sets where the analysts disagree. 
Depending on how "confident" a decision maker is in the $i^{th}$ analyst, they choose a different $A_i$, provided $(A_1, A_2, \cdots, A_n)$ form a partition of $\Omega$. 
In some sense, $A_i$ is the set where the decision maker considers that the $i^{th}$ analyst is "right" (when all analysts disagree on the level $\alpha$ to be used).
\end{itemize}
 
\subsection{A "randomised" Average Value at Risk} 

In the usual probabilistic framework,
the Average Value at Risk (AVaR) (also known as Tail Value at Risk (TVaR)) at level $\alpha \in (0,1)$ of a given \textit{potential loss} $X \in \chi(\F)$ is defined by: 
$AVaR_{\alpha}(X)=\frac{1}{1-\alpha}\int_{\alpha}^{1}q_X(t) dt $. 
The deterministic distortion function induced by $AVaR_{\alpha}$ is of the following form $\psi(x)=\frac{1}{1-\alpha}\min \left\{x,1-\alpha\right\}$ for all $x \in [0,1]$, and is a concave distortion function. 
The same sign convention in the definition of the AVaR as the one used in this paper is used, for instance, in 
\cite{embrechts2014academic}, 
\cite{denuit2006actuarial} 
or \cite{dhaene2006risk}.

\paragraph{Example 2}
Let $\G=\{\emptyset, \Omega, A, A^c\}$. 
 Let $\alpha$ and $\beta$ be two different levels for the AVaR such that $\alpha < \beta$. 
We define a random distortion function $\phi^{\G}$ as follows: for all $t \in [0,1]$,
\[
\phi^{\G}(\omega,t):= \left\{ \begin{array}{cc}
 \frac{1}{1-\alpha}\min \left\{t,1-\alpha\right\}, &  \text{if}\,\, \omega \in A\\
 \frac{1}{1-\beta}\min \left\{t,1-\beta\right\},  &  \text{if}\,\, \omega \in A^c.  
 \end{array} \right.
\]
Let $c$ be a capacity.
For the above choice of $\phi^{\G}$, we consider the randomly distorted Choquet integral 
 $\EE_{\phi^{\G} \circ c}(X)$, for  $X \in \chi$.
 With this $\phi^{\G}$, let us compute $\EE_{\phi^{\G} \circ c}(X)$ when $X$ is defined by $X:=\1_C$, with $C \in \F$.
\begin{itemize}
\item[\textbf{a.}] The probabilistic framework \\
 In the particular case where $c=P$ is a probability measure, we set: $P(C)=p$, where $p \in [0,1]$. 
We obtain: 
\[ \begin{split}
 \EE_{\phi^{\G} \circ P} (X)(\omega)&= \EE_{\phi^{\G} \circ P}(\1_C)(\omega) = \phi^{\G}(\omega,P(C))  = \phi^{\G}(\omega,p)
\\&= 
 \frac{1}{1-\alpha}\min \left\{p,1-\alpha\right\} \1_A(\omega) + 
 \frac{1}{1-\beta}\min \left\{p,1-\beta\right\} \1_{A^c}(\omega).
\end{split}\]
Based on the value of $p$ (with respect to $\alpha$ and $\beta$), we have the following three cases: 
\begin{enumerate}
    \item If $p \leq 1- \beta$, then $ \EE_{\phi^{\G} \circ P} (X)(\omega)
= \frac{p}{1-\alpha} \1_A(\omega) + \frac{p}{1-\beta} \1_{A^c}(\omega)$.  

    \item If $1-\beta < p \leq 1-\alpha$, then $ \EE_{\phi^{\G} \circ P} (X)(\omega)
= \frac{p}{1-\alpha}\1_{A}(\omega) +  \1_{A^c}(\omega) $.  

      \item If $p > 1-\alpha$, then $ \EE_{\phi^{\G} \circ P} (X)=1$, for all $\omega \in \Omega$.
\end{enumerate}

 \textbf{An application:}
The application is similar to that of Example 1 (a), 
where now the class of risk measures on which the two experts agree is the class $\{AVaR_{\alpha}: \alpha \in (0,1)\}$,
 while they disagree on the parameter to be used, with $\alpha < \beta$.
By the definition of the AVaR (in the probabilistic framework), 
when $p > 1-\alpha$,  both experts attribute a risk of $1$ to $X$.
If $1-\beta < p \leq 1-\alpha$, then they disagree: the first expert attributes $\frac{p}{1-\alpha}$ and the second expert attributes $1$. 
If $p \leq 1- \beta$, then they also disagree: the first expert attributes $\frac{p}{1-\alpha}$ but the second expert attributes $\frac{p}{1-\beta}$.
Using the randomly distorted Choquet integral, with $\phi^{\G}$ as above, allows us to "randomise" the risk measure on the sets where the two experts disagree. 

\item[\textbf{b.}] The capacity framework \\
 Let $c$ be a capacity. We set $c(C)=\gamma$, where $\gamma \in [0,1]$. 
We obtain: 
\[ \begin{split}
 \EE_{\phi^{\G} \circ c} (X)(\omega&)= \EE_{\phi^{\G} \circ P}(\1_C)(\omega) = \phi^{\G}(\omega,c(C))  = \phi^{\G}(\omega,\gamma)
\\&= 
 \frac{1}{1-\alpha}\min \left\{\gamma,1-\alpha\right\} \1_A(\omega) +
 \frac{1}{1-\beta}\min \left\{\gamma,1-\beta\right\} \1_{A^c}(\omega).
\end{split}\]


 \textbf{An application:}
The application is similar to that of Example 1 (b), 
where now the class of risk measures is $\{GAVaR_{\alpha}: \alpha \in (0,1)\}$.
  
\item [\textbf{c.}] The case where $\G=\sigma(\{A_1, A_2, \cdots, A_n\})$ 
\\
Let us consider the case where $\G=\sigma(\{A_1, A_2, \cdots, A_n\})$, where the $A_i$'s form a partition of $\Omega$. 
Let $c$ be a capacity.
Our framework allows to have different levels,
say $\alpha_1, \alpha_2, \cdots, \alpha_n$  where $\alpha_1 < \alpha_2 < \cdots < \alpha_n$. 
We define a random distortion function $\phi^{\G}$ as follows: for all $t \in [0,1]$, for all $\omega \in \Omega$, 
$\phi^{\G}(\omega,t):= \sum_{i=1}^n   \frac{1}{1-\alpha_i}\min \left\{t,1-\alpha_i\right\} \1_{A_i}(\omega)$. By our results, we know that  $\EE_{\phi^{\G} \circ c}(\cdot)$ is  monotone with respect to the stop-loss ordering (with respect to the capacity $c$), as for all $\omega\in\Omega$, $t\mapsto \phi^{\G}(\omega,t)$ is concave.    
Let us compute $\EE_{\phi^{\G} \circ c}(X)$ when $X$ is defined by $X:=\1_C$, with $C \in \F$.
 We set $c(C)=\gamma$, where $\gamma \in [0,1]$.  
We obtain: 
\begin{align*}
\EE_{\phi^{\G} \circ c}(\1_C)(\omega) = \phi^{\G}(\omega,c(C))  = \phi^{\G}(\omega,\gamma) = \sum_{i=1}^n  \frac{1}{1-\alpha_i}\min \left\{\gamma,1-\alpha_i\right\} \1_{A_i}(\omega).
\end{align*}
  \textbf{An application:}
  We have an interpretation which is analogous to that of Example 1 (c).
\end{itemize}

\subsection{A "randomised mixture" of VaR and AVaR}
For illustrative purposes,
let us now return to the simple case where $\G=\{\emptyset, \Omega, A, A^c\}$. Let us place ourselves in the capacity framework, where $c$ is a given capacity. 
 Our framework allows also  to have distortion functions belonging to two different classes (for instance, one induced by the GVaR and the other induced by the GAVaR), depending on whether $\omega$ is in $A$ or in $A^c$. In this case, the two experts disagree also on the class of risk measures to be used (and agree only on the fact that they face a random loss $X$).  
We define a random distortion function $\phi^{\G}$ as follows: for all $t \in [0,1]$,
\[
\phi^{\G}(\omega,t):= \left\{ \begin{array}{cc}
    \1_{(1-\alpha,1]}(t),   &  \text{if}\,\, \omega \in A\\
  \frac{1}{1-\alpha}\min \left\{t,1-\alpha\right\},    &  \text{if}\,\, \omega \in A^c. 
 \end{array} \right.
\]
Let us compute $\EE_{\phi^{\G} \circ c}(X)$ when $X$ is defined by $X:=\1_C$, with $C \in \F$.
We set $c(C)=\gamma$, where $\gamma \in [0,1]$.  
We have: 
\[
\EE_{\phi^{\G} \circ c}(\1_C)(\omega) = \phi^{\G}(\omega,c(C))  = \phi^{\G}(\omega,\gamma)
=\left\{ \begin{array}{cc}
    \1_{(1-\alpha,1]}(\gamma),   &  \text{if}\,\, \omega \in A\\
 \frac{1}{1-\alpha}\min \left\{\gamma,1-\alpha\right\},     &  \text{if}\,\, \omega \in A^c.
 \end{array} \right.
\]

\appendix

 \section{Appendix A} \label{sec_Appendix}

  \begin{proof}[Proof of Lemma \ref{charc_1_c}]
  Let us first prove statement $(i)$. 
We fix $x \in \mathbb{R}$ and we remark that $G_X (x) = 1 - \mathbb{E}_c \left( u\left( X \right ) \right )$ where $ u(z) :=  \mathbbm{1}_{\left( x , + \infty \right)} \left(z \right )$ which proves the desired implication as the function
$u$ is a non-decreasing.
We next prove statements $(ii)$ and $(iii)$. For $t \in (0,1)$, $(ii)$ is a direct consequence of statement $(i)$ and of the definition of $r_X^+$ and $r_Y^+$.
For $t \in (0,1)$, $(iii)$ is a consequence of statement $(i)$ and of Remark \ref{remark_44}.
\\ It is now easy to extend to $t=0$ and $t=1$. 
For $t=0$, we have $ r_{X}^-(0):= r_{X}^+(0) =\inf_{t>0} r_{X}^+(t) \leq \inf_{t>0} r^+_Y (t)=r^+_Y(0):=r^-_Y(0)$ (by using the property of the upper quantile function on $(0,1)$).
    For $t=1$, we have $ r_{X}^+(1)  := r^-_X(1) = \sup_{t < 1}r_{X}^-(t) \leq \sup_{t < 1} r^-_Y (t)=r^-_Y(1)=: r_{Y}^+(1) $, where we have used the property of the lower quantile on $(0,1)$.\qed
  \end{proof}
  
\begin{proof}[Proof of Lemma \ref{lemma_30}]
For the sake of simplicity and clarity, we consider the case when $n=2$.
In the following, we show the "only if" direction and then the "if" direction.
     \\ "$\Leftarrow$" Let $X$ and $Y$ in $\chi(\F)$ be defined as follows: 
        $X=  x_1\1_{A_1}+x_2\1_{A_2}$  and $Y= y_1\1_{A_1}+y_2\1_{A_2}$
    where 
    $x_1 \geq x_2\geq 0$,  $y_1 \geq y_2 \geq 0$,
     $A_1 \cap A_2=\emptyset$ and $ A_1 \cup A_2=\Omega$.  
   Then, $X$ and $Y$ are comonotonic. Indeed, let us consider all possible cases: 
   \begin{enumerate}
       \item If $\omega, \Bar{\omega} \in A_1$, then 
      $\left( X(\omega)-X(\Bar{\omega})\right)
       \left( Y(\omega)-Y(\Bar{\omega})\right)
       =
  \left( x_1- x_1 \right)  \left( y_1 - y_1\right)
   = 0$. 
    
        \item If $\omega \in A_1$ and  $\Bar{\omega} \in A_2$, then   
     $ \left( X(\omega)-X(\Bar{\omega})\right)
       \left( Y(\omega)-Y(\Bar{\omega})\right)
       = 
  \left( x_1 - x_2 \right)  \left( y_1 - y_2\right)
  \geq 0$. 
    
          \item If $\omega \in A_2$ and  $\Bar{\omega} \in A_1$, then 
     $ \left( X(\omega)-X(\Bar{\omega})\right)
       \left( Y(\omega)-Y(\Bar{\omega})\right)
  = \left( x_2- x_1 \right)  \left( y_2 - y_1\right)
    \geq 0$. 
         \item If $\omega, \Bar{\omega} \in A_2$, then 
      $\left( X(\omega)-X(\Bar{\omega})\right)
       \left( Y(\omega)-Y(\Bar{\omega})\right)
       =
  \left( x_2- x_2 \right)  \left( y_2 - y_2\right)
    = 0$. 

   \end{enumerate}
   Thus, in all possible cases,   $\left( X(\omega)-X(\Bar{\omega})\right)
       \left( Y(\omega)-Y(\Bar{\omega})\right)\geq 0;$ hence, the comonotonicity of $X$ and $Y$ is established. \\
        "$\Rightarrow$" Let $X$ and $Y$ be two comonotonic step functions of the form: 
        \[
        X=  b_1\1_{B_1}+b_2\1_{B_2} \,\, \,\, and \,\,\,\, Y= a_1\1_{C_1}+a_2\1_{C_2}
        \]
        where 
    $ B_1 \cap B_2= C_1 \cap C_2=\emptyset$  and  $B_1 \cup B_2= C_1 \cup C_2=\Omega$. 
    \\ Without loss of generality, we assume that $  b_1 \geq b_2\geq 0 $ and $a_1 \geq a_2 \geq 0$. We have three cases: 
     \begin{enumerate}
         \item [Case 1:] If $b_1=b_2=b$, then $X=b$. We can represent $X$ as $X=b\1_{C_1}+b\1_{C_2}$, which is the desired representation. 
         \item  [Case 2:] Similarly, if $a_1=a_2=a$, then $Y=a$, and $Y$ can be written as $Y=a\1_{B_1}+a\1_{B_2}$. Hence, the result holds. 
         \item [Case 3:] If $  b_1 > b_2\geq 0 $ and $a_1 > a_2 \geq 0$, 
     then we define the following sets:
     \[
    E_{11} =  B_1 \cap C_1,  \,\,\,\,
    E_{12} =  B_1 \cap C_2,  \,\,\,\, 
    E_{21} =  B_2 \cap C_1,  \,\,\,\, 
    E_{22} =  B_2 \cap C_2.
     \]
     Note that these sets are disjoint and form a partition of $\Omega$ (also, some of them might be empty). 
     We represent  $X$ and $Y$ using these sets: 
    \[
     X=  b_1\1_{E_{11}}+b_1\1_{E_{12}}+ b_2\1_{E_{21}}+ b_2\1_{E_{22}}
     \,\, \,\, and \,\,\,\, 
     Y=  a_1\1_{E_{11}}+a_2\1_{E_{12}}+ a_1\1_{E_{21}}+ a_2\1_{E_{22}}.
    \] 
    Since $X$ and $Y$ are comonotonic, we deduce that at least one of $E_{12}$ and $E_{21}$ is empty   
    \footnote{Indeed, if $\omega \in E_{12}$ and  $\Bar{\omega} \in E_{21}$, then $  \left( X(\omega)-X(\Bar{\omega})\right)
       \left( Y(\omega)-Y(\Bar{\omega})\right)
       =
  \left( b_1- b_2 \right)  \left( a_2 - a_1\right)
    < 0$, which contradicts the comonotonicity of $X$ and $Y$.}. 
  Therefore, we can represent $X$ and $Y$ as follows: 
  \[
     X=  x_1\1_{E_{1}}+x_2\1_{E_{2}} + x_3\1_{E_{3}}
     \,\, \,\, and \,\,\,\, 
     Y=  y_1\1_{E_{1}}+y_2\1_{E_{2}}+ y_3\1_{E_{3}}.
  \]
  where $E_1$, $E_2$ and $E_3$ are disjoint, form a partition of $\Omega$, and are defined as: $E_1:=E_{11}$, $E_3:=E_{22}$, 
  and $E_2$ is either $E_{12}$ (if $E_{12}$ is non-empty) or $E_{21}$ (if $E_{21}$ is non-empty),
  and where $x_1:= b_1, $ $y_1:= a_1,$ 
  \[
  \,\, x_2:=\left\{ \begin{array}{cc}
    b_1  &  \text{if}\, E_2=E_{12}\\
    b_2  &  \text{if} \, E_2=E_{21},
 \end{array} \right.
  \,\, x_3:= b_2, \,\,\,\, and \,\,\,\,\,\,
    \,\, y_2:= \left\{ \begin{array}{cc}
    a_2  &  \text{if}\, E_2=E_{12}\\
    a_1  &  \text{if} \, E_2=E_{21},
 \end{array} \right. \,\, y_3:= a_2. 
  \]
   \end{enumerate}\qed
\end{proof}

\section*{Declarations} 
 The authors declare that they do not have any known conflicts of interest.

\end{document}